\documentclass[a4paper,11pt,reqno]{amsart}
\usepackage[utf8]{inputenc}
\usepackage[T1]{fontenc}
\usepackage{lmodern}
\usepackage[english]{babel}
\usepackage{amsmath,a4wide,wasysym}
\usepackage{xfrac}
\usepackage{esint}
\usepackage{stmaryrd,mathrsfs,bm,amsthm,mathtools,yfonts,amssymb,color,braket,booktabs,graphicx,graphics,amsfonts}
\usepackage{latexsym,microtype,indentfirst,hyperref}
\usepackage{tikz}
\usepackage{xcolor}
\usepackage{courier}
\usepackage[colorinlistoftodos]{todonotes}

\newtheorem{theorem}{Theorem}[section]
\newtheorem{lemma}[theorem]{Lemma}
\newtheorem{proposition}[theorem]{Proposition}
\newtheorem{corollary}[theorem]{Corollary}

\theoremstyle{definition}
\newtheorem{definition}[theorem]{Definition}
\newtheorem{remark}[theorem]{Remark}

\numberwithin{equation}{section}

\newcommand\dH{{\rm dist}_{\mathscr{H}}}
\newcommand\E{{\rm E}}

\newcommand\R{\mathbb{R}}
\newcommand\HH{\mathcal{H}}

\newcommand\N{{\mathbb N}}
\newcommand\e{\varepsilon}

\newcommand\Rad{\mathrm{Rsq}}
\newcommand\rad{\mathrm{rsq}}
\newcommand\varthetad{\dot{\vartheta}}
\newcommand\varthetadd{\ddot{\vartheta}}

\newcommand\etad{\dot{\varrho}}
\newcommand\etadd{\ddot{\varrho}}
\newcommand\Isq{\mathrm{Isq}}
\newcommand\isq{\mathrm{isq}}

\newcommand{\restr}{\mathbin{\vrule height 1.6ex depth 0pt width
0.13ex \vrule height 0.13ex depth 0pt width 1.3ex}}

\title[Endpoint regularity]{Endpoint regularity for $2d$ Mumford-Shah minimizers: On a theorem of Andersson and Mikayelyan}
\author[De Lellis]{Camillo De Lellis}
\address{School of Mathematics, Institute for Advanced Study, 1 Einstein Dr., Princeton NJ 05840, USA}
\email{camillo.delellis@math.ias.edu}
\author[Focardi]{Matteo Focardi}
\address{Dipartimento di Matematica e Informatica ``Ulisse Dini'', Universit\`a degli Studi di Firenze, Viale Morgagni 67/A, 50134 Firenze, Italy}
\email{matteo.focardi@unifi.it}
\author[Ghinassi]{Silvia Ghinassi}
\address{University of Washington, Department of Mathematics, Box 354350 Seattle, WA 98195-4350, USA}
\email{ghinassi@uw.edu}

\begin{document}

\maketitle

\begin{abstract}
We give an alternative proof of the regularity, up to the loose end, of minimizers, resp. critical points of the Mumford-Shah functional when they are
sufficiently close to the cracktip, resp. they consist of a single arc terminating at an interior point.
\end{abstract}

\section{Introduction}

In this paper we study the regularity properties of the jump set of 
local minimizers of the Mumford-Shah energy on an open set $\Omega\subset\R^2$, 
which for $v\in SBV (\Omega)$ is given by
\begin{equation}\label{e:ms}
\E(v):=\int_\Omega|\nabla v|^2dx+\HH^1(S_v)\, .
\end{equation}
We say that $u\colon \Omega\to\R$ is a \emph{minimizer}
if $u\in SBV(\Omega)$, $\E(u)<+\infty$ and 
\[
\E(u)\leq \E(w)\quad \text{ whenever } \{w\neq u\}\subset\subset\Omega.
\]
For the notation and all the results concerning $SBV$ functions we refer to the book \cite{AFP00}.

The Mumford-Shah functional has been proposed by Mumford and Shah in their seminal paper \cite{MS} as a variational
model for image reconstruction. Since then, it has been widely studied in the literature, from the theoretical
side but also from the numerical and applied ones (see \cite{DG,DGCL89,Dav96,BouFraMar,DMFT05} and also the many
references in \cite[Section~4.6]{AFP00}). Starting with the pioneering work \cite{DGCL89}, the existence of minimizers has been 
proved in several frameworks and with different methods, see for instance \cite{DME92,Dav96,MadSol01}. The most 
general and successful approach is that of De Giorgi and Ambrosio via the space of \emph{special functions 
of bounded variation} ($SBV$) that works in any dimension (see \cite{DGA88,Amb90,AFP00}). 

The regularity theory has seen several contributions, both in two and several space dimensions, see 
\cite{DGCL89,Dav96,B96,AP97,AFP99,AFH,lemenant,DFR14,AFP00}.
The most important regularity problem is the famous Mumford-Shah conjecture, which states that (in $2$ dimensions) 
the closure of the jump set $\overline{S}_u$ can be described as the union of a locally finite collection of injective 
$C^1$ arcs $\{\gamma_i\}$ that can meet only at the endpoints, in which case they have to form 
triple junctions. More precisely, given any point $y\in \overline{S}_u\setminus \partial \Omega$ we only have one 
of the following three possibilities:
\begin{itemize}

\item[(a)] $y$ belongs to the interior of some $\gamma_i$ and thus $S_u$, in a neighborhood of $y$, is a single 
smooth arc; in this case $y$ is called a {\em regular point}.

\item[(b)] $y$ is a common endpoint of three (and only three) distinct arcs which form (at $y$) three equal
 angles of $120$ degrees; in this case $y$ is called a {\em triple junction}.

\item[(c)] $y$ is the endpoint of one (and only one) arc $\gamma_j$, i.e. it is a ``loose end''; in this case 
$y$ is called a {\em cracktip}.
\end{itemize}
Correspondingly, for any minimizer $u$ it is known since the pioneering work of David \cite{Dav96} that: 
\begin{itemize}
 \item[(A)] If $S_u$ is sufficiently close, in a ball $B_r (x_0)$ and in the Hausdorff distance, to a diameter 
 of $B_r (x_0)$, then in the ball $B_{\sfrac r2} (x_0)$ it is a $C^{1,\kappa}$ arc.

 \item[(B)] If $S_u$ is close to a ``spider'' centered at $x_0$, i.e. three radii of $B_r (x_0)$ meeting at $x_0$ 
 at equal angles, then in the ball $B_{\sfrac r2} (x_0)$ it consists of three $C^{1,\kappa}$ closed arcs meeting 
 at equal angles at some point $y_0 \in B_{\sfrac r4} (x_0)$.
\end{itemize}
Until recently no similar result was known in the case where $S_u$ is close to a single radius of $B_r (x_0)$, namely 
the model case of (c) above. The best result available was still due to David 
(see \cite[Theorem~69.29]{DavidBook}):
\begin{itemize}
\item[(C)] if $S_u\cap B_r (x_0)$ is sufficiently close to a single radius in the 
Hausdorff distance, then $S_u \cap B_{\sfrac r2} (x_0)$ consists of a single connected arc which joins some point
$y_0\in B_{\sfrac r4} (x_0)$ with $\partial B_{\sfrac r2} (x_0)$ and which is smooth in $B_{\sfrac r2} (x_0)
\setminus \{y_0\}$. 
\end{itemize}
However, David's result does not guarantee that such arc is $C^1$ {\em up to the loose end} 
$y_0$: in particular it leaves the possibility that the arc spirals infinitely many times around it (cf. for instance 
\cite[Section 80 pg.~571]{DavidBook}).
In this note we exclude the latter possibility and we prove an $\e$-regularity result analogous to (A) and (B)
in the remaining case of cracktips. In an unpublished manuscript , cf. \cite{DF-errato}, the first and second author claimed the following theorem.

\begin{theorem}\label{t:main}
There exist universal constants $\e, \kappa >0$ with the following property. Assume that $u$ is a local
minimizer of the Mumford-Shah functional in $B_r (x_0)$ and that $\dH (S_u, \sigma)\leq \e r$ where 
$\sigma$ is the horizontal radius $[x_0, x_0 + (r, 0)]$.
Then there is a point $y_0\in B_{\sfrac r{16}} (x_0)$ and a $C^{1,\kappa}$ function 
$\psi:[0,\sfrac r4]\to [0,\sfrac r8]$ such that
\begin{equation}\label{e:main}
S_u \cap B_{\sfrac r4} (y_0) = \big\{ y_0+ (t, \psi (t)) : t\in ]0, \sfrac r4]\big\} \cap B_{\sfrac r4} (y_0)\, .
\end{equation}
\end{theorem}

The strategy was based on a suitable linearization but, due to a sign mistake, the linear equations considered in \cite{DF-errato} were actually not the correct ones. After correcting the mistake, we found that for the new linear equations not all solutions had the appropriate decay properties and in particular that there is a slowly varying solution which is not generated by any symmetry of the original nonlinear equations (cf. the discussion below). In particular the first and second author retracted the manuscript \cite{DF-errato} as the proof was not complete. Andersson and Mikayelyan, in an independent work published around a year later attempted a somewhat analogous linearization approach, based on much earlier computations from their paper \cite{AndMyk1}. The latter reference contained as well an error and got them to consider a wrong linearized problem (see the comments to \cite{AndMyk1}, which was retracted by the authors, and to version 1 of \cite{AndMyk2}). The works \cite{DF-errato} and \cite{AndMyk1}-\cite{AndMyk2} linearize the problem in a different system of coordinates, however both errors lead to a wrong sign in a corresponding term of the equations, and in particular the slowly varying solution causing troubles is the same one. After realizing their mistake, Andersson and Mikayelyan came up with a clever competitor argument which excludes that the linearization of minimizers might be (a multiple of) the problematic slowly varying solution. In particular, in a revised version of their paper, they claimed an adjusted proof of Theorem \ref{t:main}. Furthermore, since the remaining linear modes have a faster decay, they indeed claimed that the arc is $C^{2,\kappa}$ and that the curvature vanishes at the tip.

\begin{theorem}\label{t:main-C2}
Let $u$ be as in Theorem \ref{t:main}. Then the function $\psi$ is $C^{2,\kappa}$ and its second derivative vanishes at $0$.
\end{theorem}

While we are not able to follow all the details of their arguments, we believe that this is just due to some technical problems and that their approach is overall correct. In this paper we propose some alternative methods that take care of a few technical issues and give different proofs of the Theorems \ref{t:main} and \ref{t:main-C2}. In particular, even though the fundamental reason for the regularity is still the same, we use three distinct ideas:
\begin{itemize}
\item[(i)] First of all, in order to exclude the slowly varying solution of the linearized problem we use a suitable cut-off argument on the usual Euler-Lagrange identity for inner variations. This replaces computations of Andersson and Mikayelyan (which are rather involved) with a clean simple identity between boundary integrals for critical points. In particular suitable variants of Theorem \ref{t:main} and Theorem \ref{t:main-C2} are in fact valid for critical points. Our point of view gives an interesting link to another identity discovered by David and L\'eger, and Maddalena and Solimini. The one used in this paper and the one of David-L\'eger-Maddalena-Solimini are in fact particular cases of a more general family of identities found by applying the ``truncation method'' to isometries and conformal transformations.
\item[(ii)] Rather than linearizing the equations for the Mumford-Shah minimizers, we linearize the one for its harmonic conjugate. This has the advantage that the Neumann boundary condition is replaced by the Dirichlet boundary condition, simplifying several computations.
\item[(iii)] Finally, we take the change of variables approach of \cite{DF-errato}, inspired by the pioneering work of Leon Simon \cite{Simon83}, which avoids any discussion of the behavior of the solution at the tip, prior to knowing Theorem \ref{t:main}, and the technicalities involved by comparing functions defined on different domains. 
\end{itemize}

As remarked in \cite{DF-errato}, a consequence of Theorem~\ref{t:main} is a strengthening of the conclusions in
\cite[Proposition~5]{DLF13} that yields an energetic characterization of the Mumford-Shah conjecture.
\begin{proposition}\label{p:Econj}
The Mumford-Shah conjecture holds true for a local minimizer $u$ in $\Omega$
if and only if $\nabla u\in L^{4,\infty}_{loc}(\Omega)$, i.e. if
for all $\Omega^\prime\subset\subset\Omega$ there is a constant $K=K(\Omega^\prime)>0$ such that for all $\lambda> 0$
\[
|\{x\in\Omega^\prime:\,|\nabla u(x)|>\lambda\}|\leq K\lambda^{-4}.
\]
\end{proposition}

\section{Reduction to a single connected curve}

Theorems~\ref{t:main} and \ref{t:main-C2} will be proved combining (a more precise version of) David's statement (C) with
Theorem~\ref{t:main2} below, which for simplicity we state when the domain $\Omega$ is the unit disk
(a corresponding version for $\Omega = B_r (x)$ can be then proved by a simple scaling argument). Theorem~\ref{t:main2} is not comparable to Theorem \ref{t:main} because:
\begin{itemize}
\item on the one hand it assumes the stronger property that the jump set of the critical point is a single arc with one endpoint at the boundary $\partial B_1$ and the other endpoint at the origin;
\item on the other hand it assumes that $(u, S_u)$ is a critical point, rather than a minimizer.
\end{itemize}
We recall that a critical point $(u, S_u)$ satisfies the following two identities, which we will call, respectively, outer and inner variational identities:
\begin{align}
&\int_{\Omega\setminus S_u} \nabla u \cdot \nabla \varphi = 0 \qquad \forall \varphi \in C^1_c (\Omega, \mathbb R)\label{e:outer}\\
&\int_{\Omega \setminus S_u} (|\nabla u|^2 \operatorname{div} \eta - 2 \nabla u^T\cdot \nabla \eta\cdot \nabla u) = - \int_{S_u} e^T \cdot \nabla \eta\cdot  e \, d\mathcal{H}^1
\quad\forall \eta\in C^1_c (\Omega, \mathbb R^2)\, ,\label{e:inner}
\end{align}
where $e(x)$ is a unit tangent vector field to the rectifiable set $S_u$.
The first identity corresponds to the stationarity of the energy with respect to the perturbation $u_t (x) := u(x) + t \varphi (x)$.
For the second, if we define $\Phi_t (x) := x+t \eta (x)$, then $\Phi_t$ is a diffeomorphism of $\Omega$ onto itself for all sufficiently small $t$ and \eqref{e:inner} is equivalent to the condition 
$\left. \frac{d}{dt} E (u \circ \Phi_t)\right|_{t=0} = 0$. We note in passing that arguing by density \eqref{e:outer} is in fact valid for every $\varphi \in H^1_0 (\Omega)$\footnote{{We use the standard notation $H^k$ for $W^{k,2}$, for $k\in\mathbb{N}$}}, while \eqref{e:inner} for every $\eta\in C^1_0 (\Omega, \mathbb R^2)$. In fact \eqref{e:inner} can be extended to $\eta\in W^{1,\infty}_0 (\Omega, \mathbb R^2)$, but such extension would require a discussion of how to interpret appropriately the derivative of a Lipschitz function along tangent fields to rectifiable sets, which is not needed for our purposes. 

\begin{theorem}\label{t:main2}
There are universal constants $\e_0, \delta_0, C >0$ with the following property.
Let $u$ be a critical point of the Mumford-Shah functional in $B_1$ whose singular set $S_u$ is given, in cartesian coordinates, by
\begin{equation}\label{e:crack-tip}
S_u = \big\{ r (\cos \alpha (r), \sin \alpha (r)) : r\in ]0,1[\big\}
\end{equation}
for some smooth function $\alpha: ]0,1[\to \mathbb R$ with 
\begin{equation}\label{e:piccolezza}
\sup_r (r |\alpha' (r)|+ r^2 |\alpha'' (r)|) \leq \e_0\, .
\end{equation}
Then the curvature $\kappa (r)$ of the curve $S_u$ at the point $r (\cos \alpha (r), \sin \alpha (r))$ satisfies the estimate
\begin{equation}\label{e:main_estimate}
|\kappa (r)| \leq C r^{\delta_0}\, .
\end{equation}
\end{theorem}
\eqref{e:main_estimate} is easily seen to give a $C^{2}$ estimate for the curve $S_u$ and shows as well that the curvature vanishes at the tip. 

We next introduce the harmonic conjugate of $u$. First of all, given a vector $x= (x_1, x_2)\in \mathbb R^2$, we denote by $x^\perp$ its counterclockwise rotation of $90$ degrees,
$x^\perp := (-x_2, x_1)$. 
Second, observe that, by \eqref{e:outer}, the $L^2$ vector field $X := \nabla u^\perp$ is curl-free in the distributional sense, hence $\nabla u^\perp$ is the gradient of a function $w$ which is unique up to addition of a constant\footnote{Note indeed that \eqref{e:outer} is equivalent to $\int X \cdot (\nabla \varphi)^\perp =0$. Fix a family of standard mollifiers $\varphi_\varepsilon$ and observe that, for every fixed $\sigma<1$, provided $\varepsilon>0$ is sufficiently small, $X_\varepsilon:= X * \varphi_\varepsilon$ is smooth in $B_\sigma$ and satisfies $\int \varphi\, {{\textrm{curl}}\, X_\varepsilon} = - \int X_\varepsilon \cdot \nabla^\perp \varphi = 0$ for every $\varphi \in C^\infty_c (B_\sigma)$. This shows that $X_\varepsilon$ is curl-free in the classical sense and hence the gradient of a smooth potential $w^\sigma_\varepsilon$, which we can assume to have average zero in $B_\sigma$. By $H^1$ compactness, as $\varepsilon \downarrow 0$, $w^\sigma_\varepsilon$ converges, up to subsequences, to a $H^1$ potential $w^\sigma$ for $X$ on $B_\sigma$. Letting $\sigma\uparrow 1$ and again using standard Sobolev space theory, we conclude the existence of $w\in H^1$ such that $\nabla w = X$.}. Next note that, since $S_u$ is smooth on $B_1\setminus \{0\}$ and, by
\eqref{e:outer}, $u$ and $w$ are harmonic on $B_1\setminus S_u$, they both have smooth traces on $S_u\setminus \{0\}$ from both sides of $S_u$. Being $w$ in $H^1$ the two traces of $w$ agree on $S_u\setminus \{0\}$ and thus $w$ is continuous on $S_u\setminus\{0\}$ and therefore in $B_1\setminus \{0\}$. Moreover, again by \eqref{e:outer}, $u$ satisfies the Neumann boundary condition on $S_u\setminus \{0\}$, which in turn implies that the tangential derivative of the traces of $w$ along $S_u\setminus \{0\}$ vanishes. Thus $w$ is constant on $S_u\setminus \{0\}$ and, by possibly adding a constant to it, we fix its value on the curve to be equal to $0$. Finally, note that we can apply Bonnet's monotonicity formula, {{that it is valid in particular for critical points (cf. \cite[Theorem~3.1]{B96}),} and thus conclude that
\[
\int_{B_r} |\nabla w|^2 = \int_{B_r\setminus S_u} |\nabla u|^2 \leq r \int_{B_{1}} |\nabla u|^2\, . 
\]
In particular a simple scaling argument using regularity of harmonic functions implies that $\|\nabla w\|_{L^\infty (B_r\setminus B_{r/2})} \leq C r^{-1/2}$. Hence $w$ extends continuously to the point $0$ (and in fact it belongs to $C^{1/2}_{loc} (B_1)$). We can restate Theorem~\ref{t:main2} using the harmonic conjugate as follows:

\begin{theorem}\label{t:main3}
There are universal constants $\e_0, \delta_0, C >0$ with the following property. Let $\alpha: ]0,1[\to \mathbb R$ be a smooth function with 
$\sup_r (r |\alpha' (r)|+ r^2 |\alpha'' (r)|) \leq \e_0$ and set
\begin{equation}\label{e:crack-tip-2}
K := \big\{ r (\cos \alpha (r), \sin \alpha (r)) : r\in ]0,1[\big\}\, .
\end{equation}
Let $w\in H^1\cap C^{1/2} (B_1)$ be such that $\Delta w=0$ on $B_1\setminus K$, $w|_K =0$ and the identity
\begin{equation}\label{e:inner-2}
 \int_{\Omega \setminus K} (|\nabla w|^2 \operatorname{div} \eta - 2 (\nabla w^\perp)^T\cdot \nabla \eta\cdot \nabla w^\perp) = - \int_K e(x)^T \cdot \nabla\eta\cdot  e(x) \, d\mathcal{H}^1(x)
\end{equation}
holds for every $\eta\in C^1_0 (B_1, \mathbb R^2)$.
Then \eqref{e:main_estimate} holds. 
\end{theorem}

\section{Singular inner variations} 

Clearly if $\eta$ is not compactly supported, the identity \eqref{e:inner} is not valid any more. However, consider the case in which $B_1 \subset \Omega$. We can then take a 
sequence of cut-off functions $\varphi_k \in C^\infty_c (B_1)$ with the property that $\varphi_k \uparrow \mathbf{1}_{B_1}$. Hence we can plug $\varphi_k \eta$ into \eqref{e:inner} and derive, for $k\uparrow \infty$, the analog of \eqref{e:inner}, which results into the same identity with some additional boundary terms.

This procedure was first applied by Maddalena and Solimini in \cite{MS} to the vector field $\eta (x) = x$ to derive an interesting identity discovered independently by David and L\'eger in \cite{DL02}. The general statement is

\begin{theorem}\label{t:boundary-variations}
Let $u$ be a critical point of the Mumford-Shah functional in $\Omega$ such that:
\begin{itemize}
\item[(a)] $B_r \subset \Omega$;
\item[(b)] $S_u\cap \partial B_r$ is a subset of the regular part of $S_u$;
\item[(c)] $S_u$ intersects $\partial B_r$ transversally in a finite number of points.
\end{itemize}
Let $\eta\in C^1 (\overline{B}_r, \mathbb R^2)$. Then:
\begin{align} 
&\int_{B_r \setminus S_u} (|\nabla u|^2 \operatorname{div} \eta - 2 \nabla u^T\cdot \nabla \eta\cdot \nabla u) + \int_{B_r\cap S_u} 
e^T \cdot \nabla\eta\cdot  e \,d\mathcal{H}^1
\nonumber\\
= &\int_{\partial B_r \setminus S_u} \left(|\nabla u|^2 \eta \cdot \nu - 2\frac{\partial u}{\partial \nu}
\nabla u\cdot\eta\right) d\mathcal{H}^1 +
\sum_{p\in S_u \cap \partial B_r} e(p) \cdot \eta (p)\, ,\label{e:int-variation-bdry} 
\end{align}
where  $\nu (x) = \frac{x}{|x|}$ is the exterior unit normal to the circle and $e(p)$ is the tangent unit vector to $S_u$ such that $e (p) \cdot p = |p| (e (p)\cdot \nu (p)) >0$.
\end{theorem}
\begin{proof}
Fix $\eta\in C^1 (\overline{B}_r, \mathbb R^2)$  and 
$\varphi_k\in C_c^{\infty}(B_r)$. Since $\nabla (\varphi_k\eta)=\varphi_k \nabla \eta+\eta\otimes\nabla\varphi_k$, the identity \eqref{e:inner} applied to the test field $\varphi_k\eta\in C^1_c (B_r, \mathbb R^2)$ can be rewritten as 
\begin{align} \label{e:int-variation bis} 
\int_{B_r\setminus S_u} &(|\nabla u|^2 \operatorname{div} \eta - 2 \nabla u^T\cdot \nabla \eta\cdot \nabla u) \varphi_k
+\int_{S_u\cap B_r} \big(e^T \cdot \nabla\eta\cdot  e\big)\varphi_k \, d\mathcal{H}^1\notag\\
=&- \int_{B_r\setminus S_u} \big(|\nabla u|^2 \nabla\varphi_k \cdot\eta- 2 \big(\nabla\varphi_k\cdot \nabla u\big)
\big(\nabla u\cdot \eta\big)\big)
-\int_{S_u\cap B_r} \big(e \cdot \eta\big)\big(e\cdot\nabla \varphi_k\big)\, d\mathcal{H}^1\, .
\end{align}
In order to simplify further our computations assume additionally that $\varphi_k (x) = g_k (|x|)$ for some smooth function $g_k$ of one variable such that
$g_k\equiv 1$ on $[r (1-\frac{1}{k}), r]$,
$|g'_k|\leq 2k r^{-1}$ and $g_k \uparrow \mathbf{1}_{[0,1[}$. Thus the left hand side of \eqref{e:int-variation bis} converges to the first line of \eqref{e:int-variation-bdry} by dominated convergence. Observe next that $\nabla\varphi_k\stackrel{*}{\to} \mu = -\nu\mathcal{H}^1\restr\partial B_r$ in the sense of measures on $\Omega$ and
that, although $\nabla u$ is discontinuous on $S_u$, it is right and left continuous on it in the corona $B_{r(1+\sigma)}\setminus B_{r(1-\sigma)}$ for any sufficiently small $\sigma$. Since $|\mu| (S_u) =0$ the first integral in the right hand side of \eqref{e:int-variation bis} converges to the first integral in the second line of \eqref{e:int-variation-bdry}. As for the last term, enumerate the points $\{p_1, \ldots, p_N\}= S_u \cap \partial B_r$. For every sufficiently small $\sigma$ the set $(B_r\setminus \overline{B}_{r(1-\sigma)}) \cap S_u$ consists of finitely many connected components $\gamma_1, \ldots , \gamma_N$ such that $p_i = \overline{\gamma}_i \cap \partial B_r$. We next compute
\[
\lim_{k\to \infty} \underbrace{\int_{\gamma_i} \big(e \cdot \eta\big)
\big(e\cdot\nabla \varphi_k\big)\, d\mathcal{H}^1}_{=: I_k (i)}\, .
\]
Observe that for each $t\in [r(1-\frac{1}{k}), r]$ and each $i\in \{1, \ldots , N\}$, $\gamma_i\cap \partial B_t$ consists of a single point $p_i (t)$. Moreover choose $e (p)$ on each $\gamma_i$ in such a way that it is continuous and $e (p_i) \cdot \nu (p_i) = e (p_i) \cdot p_i >0$. Using the coarea formula with the function
$d (x) = |x|$ we get
\begin{align*}
I_k (i) = \int_{r (1-\frac{1}{k})}^r \big(e(p_i (t)) \cdot \eta(p_i (t))\big)\left(e(p_i (t))\cdot \frac{p_i (t)}{|p_i (t)|}\right) g'_k (t) \frac{1}{|\nabla d (p_i (t)) \cdot e (p_i (t))|}\, dt\, .
\end{align*}
Observe that $
\nabla d (p_i (t)) \cdot e (p_i (t)) = e(p_i (t))\cdot \frac{p_i (t)}{|p_i (t)|}$
and that the latter is a positive number if $p_i (t)$ is sufficiently close to $p_i (1)=p_i$. Moreover the function $h(t) = \big(e(p_i (t)) \cdot \eta(p_i (t))\big)$ is continuous and $h(1) = e (p_i) \cdot \eta (p_i)$. Therefore we have
\begin{align*}
I_k (i) & = \int_{1-\frac{1}{k}}^1  \big(e(p_i (t)) \cdot \eta(p_i (t))\big) g_k' (t)\, dt\\
&= e (p_i) \cdot \eta (p_i) \int_{1-\frac{1}{k}}^1 g_k' (t)\, dt + \int_{1-\frac{1}{k}}^1  (h(t)-h(1)) g_k' (t)\, .
\end{align*}
Recalling that $\int_{1-\frac{1}{k}}^1 g_k' (t)\, dt = -1$ and that $\|g'_k\|_0 \leq 2k$, the continuity of $h$ shows that $\lim_k I_k (i) = - e (p_i)\cdot \eta (p_i)$. 
\end{proof}

There are two interesting particular classes of vector fields that one could use as tests in \eqref{e:int-variation-bdry}. First of all, if $\eta$ is conformal, then $|v|^2 {\rm div}\, \eta (x) - 2 v^T \cdot \nabla \eta (x) \cdot v=0$ for every $x\in B_1$ and for every $v\in \mathbb R^2$. Therefore the first bulk integral in \eqref{e:int-variation-bdry} disappears. A very particular family of conformal vector fields $\eta$ is of course given by constant vector fields and rotations: for the first $\nabla\eta$ vanishes and for the second $\nabla\eta$ is a (constant) skewsymmetric matrix and thus $v^T\cdot \nabla \eta \cdot v$ vanishes for every vector $v\in \mathbb R^2$.  We thus derive the following simple corollary of Theorem \ref{t:boundary-variations}

\begin{corollary}\label{c:boundary-variations}
Let $u$ be a as in Theorem \ref{t:boundary-variations}.
If $\eta\in C^1 (\overline{B}_r, \mathbb R^2)$ is conformal, then
\begin{align} 
\int_{B_r\cap S_u} e^T \cdot \nabla \eta\cdot  e \, d\mathcal{H}^1
= \int_{\partial B_r \setminus S_u} \left(|\nabla u|^2 \eta \cdot \nu - 2\frac{\partial u}{\partial \nu}
\nabla u\cdot\eta\right)d\mathcal{H}^1 +
\sum_{p\in S_u \cap \partial B_r} e(p) \cdot \eta (p)\, .\label{e:conformal} 
\end{align}
In particular, for every constant vector $v$ we have 
\begin{equation}\label{e:translations}
0 = \int_{\partial B_r \setminus S_u}  \left(|\nabla u|^2 v \cdot \nu - 2\frac{\partial u}{\partial \nu}\frac{\partial u}{\partial v}\right)d\mathcal{H}^1 +
\sum_{p\in S_u \cap \partial B_r} e(p) \cdot v\, 
\end{equation}
and if $\tau$ denotes the tangent to the unit circle, then
\begin{equation}\label{e:rotations}
0=\sum_{p \in S_u \cap \partial B_r} e(p) \cdot \tau(p) - 2 \int_{\partial B_r\setminus S_u}\frac{\partial u}{\partial \nu}\frac{\partial u}{\partial \tau}d\mathcal{H}^1\,.
\end{equation}
\end{corollary}

The David-L\'eger-Maddalena-Solimini identity is given when $\eta(x)=x$ in \eqref{e:conformal}:
\begin{equation}\label{e:DLE}
\frac{1}{r} \mathcal{H}^1 (S_u\cap B_r) = \int_{\partial B_r\setminus S_u} \left(\left(\frac{\partial u}{\partial \tau}\right)^2 - \left(\frac{\partial u}{\partial \nu}\right)^2\right)d\mathcal{H}^1
+ \sum_{p\in S_u \cap \partial B_r} e(p) \cdot 
\nu(p)\, .
\end{equation}
Next, consider the situation in which $S_u\cap \partial B_1$ consists of a single point $p$. We can then take a suitable linear combination of \eqref{e:translations} and \eqref{e:rotations} to derive a boundary integral identity which does not involve the set $S_u$. 

\begin{corollary}\label{c:AM-variations}
Let $u$ be as in Theorem \ref{t:boundary-variations} and assume that $S_u\cap \partial B_r$ is a singleton $\{p\}$. Then, 
\begin{equation}\label{e:AM-identity}
\int_{\partial B_r\setminus \{p\}} \left(|\nabla u (q)|^2 \nu (q) \cdot \tau (p) + 2 \frac{\partial u}{\partial \nu} (q) \nabla u (q) \cdot \left(\tau (q) - \tau (p)\right)\right)\, d\mathcal{H}^1 (q) = 0\, .
\end{equation}
\end{corollary}

We will use the latter identity to exclude the ``lowest mode'' in the series expansion of solutions of the linearized equation \eqref{e:lineare}, cf. Section \ref{s:tre-anelli}. This, loosely speaking, corresponds to the competitor argument used by Andersson and Mikayelyan in \cite{AndMyk2} to exclude a similar term in the linearization considered there. Its advantage over the argument used in \cite{AndMyk2} is, however, that only boundary integrals of the actual critical points are involved and we do not need to discuss any harmonic extension of competitors.

\section{Rescaling and reparametrization}\label{s:exp-rep}

Before starting our considerations, we must introduce the model ``tangent function'' of a local minimizer at a loose end,
which in polar coordinates is given by
\begin{align}
\Rad(\phi, r) &:=\sqrt{\textstyle{\frac{2r}{\pi}}}\cos\textstyle{\frac{\phi}{2}}\label{e:Rad}\, 
\end{align}
and whose singular set $S_{\Rad}$ is the open half line $\{(t, 0): t\in \R^+\}$ (in cartesian coordinates). Observe that $\Rad$ is, up to the prefactor $\sqrt{\frac{2}{\pi}}$, the real part of a branch of the complex square root. We will likewise use the notation $\Isq$ for the imaginary part of the same branch multiplied by the same prefactor, namely
\begin{equation}\label{e:Isq}
\Isq (\phi, r) := \sqrt{\textstyle{\frac{2r}{\pi}}}\sin\textstyle{\frac{\phi}{2}}\, .
\end{equation}

It was conjectured by 
De Giorgi that $\Rad$ is the unique {\em global minimizer} in $\mathbb R^2$. In particular its restriction to any bounded open set is a 
minimizer in the sense introduced above. This last property was proved in a remarkable book by Bonnet and David, see \cite{BDBook}.

\subsection{Rescalings}

From now until	 the very last section, $u$ will always denote a critical point of the Mumford-Shah energy 
in $B_1$ satisfying the assumptions of Theorem~\ref{t:main2}. Keeping the notation introduced there, for 
$\rho>0$ set
\begin{align}
u^\rho(\phi, r)&:=\rho^{-\sfrac12}\,u(\phi+ \alpha(\rho\,r), \rho\, r),\label{e:riscala1}\\
\alpha^\rho (r) &:= \alpha (\rho\, r)\, .\label{e:riscala2}
\end{align}
\begin{lemma}\label{l:Bonnet}
 For every $\delta>0$ and for every $k\in\N$ there is $\e_1>0$ such that if $u$ and $\alpha$ are as in
 Theorem~\ref{t:main2} with $\e_0\leq\e_1$, then
 \begin{equation}\label{e:Bonnet0}
 \|u^\rho-\Rad\|_{C^k([0,2\pi]\times[\sfrac12,2])} 
+ \|\alpha^\rho\|_{C^k ([\sfrac 12, 2])}\leq\delta\qquad 
 \forall\,\rho \leq \frac{1}{4}\, .
\end{equation}
\end{lemma}
\begin{proof}
 The statement follows easily from the blow-up technique of Bonnet, see \cite{B96}, and the higher 
 differentiability theory of \cite{AFP99}. 
\end{proof}
\begin{corollary}\label{c:Bonnet}
 For every $\delta>0$ and for every $k\in\N$ there is $\e_1>0$ with the following property. 
 If $u$ and $\alpha$ satisfy the assumptions of Theorem~\ref{t:main2} with $\e_0\leq\e_1$, then
 \begin{equation}\label{e:Bonnet2}
\sup_{[0,2\pi]\times]0,\sfrac12[}
r^{i-\sfrac12}|\partial_\phi^{j}\partial_r^{i}\big(u(\phi + \alpha (r) ,r)-\Rad(\phi,r)\big)|
\leq \delta\quad \text{$\forall\, i+j\leq k$},
\end{equation}
\begin{equation}\label{e:Bonnet2bis}
\sup_{]0,\sfrac 12[}r^i|\alpha^{(i)}(r)|\leq \delta \quad\qquad\qquad\qquad\qquad \forall\, i\leq k\, . 
\end{equation}
\end{corollary}
\begin{proof}
Observe first that
\[
(\alpha^\rho)^{(i)}  (r) = \rho^i \alpha^{(i)} (\rho r)\, .
\]
Taking the supremum in $r\in [\sfrac 12,2]$ in the latter identity, we easily infer 
\[
\rho^i \|\alpha^{(i)}\|_{C^0 ([\sfrac \rho 2, 2\rho])} = \|(\alpha^{\rho})^{(i)}\|_{C^0 ([\sfrac 12, 2])}\,, 
\]
and hence conclude \eqref{e:Bonnet2bis} 
from Lemma~\ref{l:Bonnet}. 

Next, from \eqref{e:riscala1} and the $\sfrac{1}{2}$-homogeneity of $\Rad$ we conclude
\begin{align*}
 u (\phi + \alpha (r) , r) - \Rad (\phi, r) = \rho^{\sfrac{1}{2}} \left( u^\rho \left(\phi, \textstyle{\frac{r}{\rho}}\right)
- \Rad \left(\phi, \textstyle{\frac{r}{\rho}}\right)\right)\, .
\end{align*}
Differentiating the latter identity $j$ times in $\theta$ and $i$ times in $r$, we conclude
\begin{align*}
\partial_r^i\partial_\phi^j\left(u(\phi+\alpha(r) , r) -\Rad (\phi, r)\right)
= \rho^{\sfrac{1}{2} -i} \partial_r^i\partial_\phi^j \left(u^\rho - \Rad\right) \left(\phi, \textstyle{\frac{r}{\rho}}\right)
\end{align*}
Substitute first $\rho=r$ and take then the supremum in 
$\phi$ and $r$ to achieve  \eqref{e:Bonnet2}, 
again from Lemma~\ref{l:Bonnet}.
\end{proof}

\subsection{Reparametrization}

We next introduce the functions 
\begin{align}
\vartheta(t):=&\alpha(e^{-t}) = \alpha^{e^{-t}} (1)\label{e:vartheta}\\
\varrho(t) :=& e^{-t}\big(\cos\vartheta(t),\sin\vartheta(t)\big)\label{e:varrho}\\
 f (\phi, t):=& e^{\sfrac{t}2} w \big(\phi+\vartheta(t), e^{-t}\big) 
=w^{e^{-t}} (\phi, 1)\, ,\label{e:f}\\
\rad (\phi) := &\Rad (\phi, 1).\label{e:radino}\\
\Isq(\phi, r) &:= \sqrt{\textstyle{\frac{2r}{\pi}}}
\sin\left(\textstyle{\frac{\phi}{2}}\right)\label{e:Radw}\\
\isq (\phi) := &\Isq (\phi, 1).\label{e:radinow}
\end{align}

In the next lemma we derive a system of partial differential equations for the functions $f$ and $\vartheta$, 
exploiting the Euler-Lagrange conditions satisfied by $u$ and $S_u$ (cf. \eqref{e:outer} and \eqref{e:inner}). 
We also rewrite the estimates of Corollary~\ref{c:Bonnet} in terms of the new functions.
It is more convenient to work with $w$ rather than $u$, because of the homogeneous Dirichlet boundary condition satisfied by $w$ on $S_u$ instead of its Neumann counterpart satisfied by $u$.

\begin{lemma}\label{l:nonlinear}
 If $u$ satisfies the assumptions of Theorem~\ref{t:main2} and $\vartheta, f$ are given by \eqref{e:vartheta} 
 and \eqref{e:f}, then
\begin{equation}
  \begin{cases}\label{e:SIS}
 \displaystyle{f_t =\frac f4+ f_{\phi\phi}+ f_{tt}+\big(\varthetad f_\phi 
 +\varthetad^2 f_{\phi\phi}-2\varthetad f_{t\phi}-\varthetadd f_\phi \big)}\cr\cr
f(0,t)=f(2\pi,t)=0
 \cr\cr
 \displaystyle{ \frac{\varthetadd - \varthetad-\varthetad^3}{(1+\varthetad^2)^{\sfrac52}}
 =f_\phi^2(2\pi,t)-f_\phi^2(0,t)
%
\,.}
 \end{cases}
\end{equation}
Moreover, for every fixed $\sigma, \delta>0$ and $k\in\N$, the following estimates hold provided $\e_0$ in 
Theorem~\ref{t:main2} is sufficiently small:
\begin{equation}\label{e:decay_g}
 \|\vartheta^{(i)}\|_{C^0([\sigma , \infty[)}\leq\delta \qquad \mbox{for all $i\leq k$,}
\end{equation}
\begin{equation}\label{e:decay_f}
\|\partial_\phi^i\partial_t^j (f - \isq)\|_{C^0([0,2\pi]\times[\sigma,\infty[)} \leq \delta\,
\qquad\mbox{for all 
$i+j\leq k$}.
\end{equation}
\end{lemma}
\begin{proof}
 Let us first introduce the unit tangent and normal vector fields to $S_u$ denoted by $e (t)$ 
 and $n (t)$:
 \begin{align*}
  e (t):= \frac{\etad(t)}{|\etad(t)|},\quad n (t):=e^\perp(t).
 \end{align*}
Moreover, we will denote by $\nabla u^+$ and $\nabla u^-$ the traces of $\nabla u$ on $S_u$ where $\pm$ is identified by the direction in which the vector $n$ is pointing. More precisely, if $p\in S_u$, then
\begin{align*}
\nabla u^+ (p) &= \lim_{s\downarrow 0} \nabla u (p+ sn (p))\,,\\
\nabla u^- (p) & = \lim_{s\downarrow 0} \nabla u (p-sn(p))\, .
\end{align*}
Observe that, under the assumptions of Lemma \ref{l:Bonnet}, $e (t)$ is pointing ``inward'', i.e. towards the origin, and hence
for $p = \varrho (t) = (e^{-t} (\cos (\vartheta (t)), \sin (\vartheta (t)))$ we have
\begin{align}
\nabla u^+ (p) &= \lim_{\phi \uparrow 2\pi} \nabla u (e^{-t} (\cos (\vartheta (t) + \phi), \sin (\vartheta (t) +\phi))\, ,\\
\nabla u^- (p) &= \lim_{\phi\downarrow 0} \nabla u (e^{-t} (\cos (\vartheta (t) + \phi), \sin (\vartheta (t) +\phi))\, .
\end{align}
We refer to Figure \ref{figura-1} for a visual illustration.

\begin{figure}
\centering
\begin{tikzpicture}
\draw[>=stealth,->] (-0.5,0) -- (3.5,0);
\node[below] at (3.5,0) {$x_1$};
\draw[>=stealth,->] (0,-0.5) -- (0,2.9);
\node[left] at (0,2.9) {$x_2$};
\draw[very thick] (0,0) to [out=0, in=200] (1.5,0.7) to [out=20, in=180] (2.7, 0.3) to [out = 0, in=225] (3.3,0.6);
\draw[>=stealth,->] (1.5,0.7) -- ({1.5-1.2*cos(20)},{0.7-1.2*sin(20)});
\draw[fill] (1.5,0.7) circle [radius=0.05];
\draw[>=stealth,->] (1.5,0.7) -- ({1.5+1.2*sin(20)},{0.7-1.2*cos(20)});
\node[right] at (1.5,1.1) {$p = \varrho (t)$};
\node[above] at ({1.5-1.2*cos(20)},{0.7-1.2*sin(20)}) {$e(p)$};
\node[below] at  ({1.5+1.2*sin(20)},{0.7-1.2*cos(20)}) {$n(p)$};
\node[below left] at (1.5,0.5) {$+$};
\node[above left] at (1.5,0.7) {$-$};
\end{tikzpicture}
\caption{The tangent vector $e(p)$ and the normal vector $e(p)$ and a point $p\in S_u$. Since $t\mapsto |\varrho (t)|$ is a decreasing function, $e(p)$ points towards the origin. Consequently the convention for the symbols $\pm$ on traces of functions is as illustrated in the picture.}
\label{figura-1}
\end{figure}
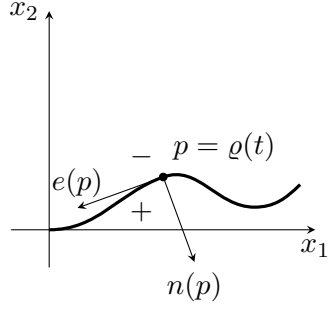

Since $u\colon B_1\subset\R^2\to\R$ is a critical point of the E energy, 
\begin{equation}\label{e:firstvar}
 \begin{cases}
 \triangle u=0 & \mbox{on } B_1\cr
 \frac{\partial u}{\partial n}=0 & \mbox{on } S_u\cr
 \mathbf{k}=-|\nabla u^+|^2+|\nabla u^-|^2\qquad&\mbox{on }  S_u\, .
 \end{cases}
\end{equation}
where $\mathbf{k}$ is the curvature of $S_u$ given by 
\begin{align*}
{\bf k} =& \frac{1}{|\etad(t)|}\dot e(t)\cdot n(t)\, .
\end{align*}
In particular, the harmonic conjugate $w$ of $u$ satisfies
\begin{equation}\label{e:firstvar2}
 \begin{cases}
 \triangle w=0 & \mbox{on } B_1\cr
 w=0 & \mbox{on } S_u\cr
 \mathbf{k}=-|\nabla w^+|^2+|\nabla w^-|^2\qquad&\mbox{on }  S_u\, .
 \end{cases}
\end{equation}
Recalling that
 \begin{equation}\label{e:ftr}
 w(\phi ,r)=r^{\sfrac12}f(\phi-\vartheta(-\ln r) , - \ln r),
  \end{equation}
we compute
\begin{equation}\label{e:utr}
 w_r =r^{-\sfrac12}\left(\frac f2- f_t +\varthetad f_\phi\right),\quad 
 w_\phi =r^{\sfrac 12} f_\phi .
 \end{equation}
Next we recall the formula for the Laplacian in polar coordinates:
\[
\triangle w=0\quad\Longleftrightarrow\quad r^{-2} w_{\phi\phi} +r^{-1} (r w_r )_r =0.
\]
By means of \eqref{e:utr} we get
\[
 r^{-2} w_{\phi\phi} =r^{-\sfrac32} f_{\phi\phi}\, ,
\]
and
\begin{align*}
 r^{-1} (r w_r )_r
=& r^{-1} \left(r^{\sfrac12}\left(\frac f2- f_t +\varthetad f_\phi\right)\right)_r\\
=& r^{-\sfrac32}\left(\frac f4-\frac{f_t}{2}+\frac{\varthetad f_\phi}{2} \right)
+r^{-\sfrac12}\left(- r^{-1}\frac{f_t}{2} +r^{-1} \varthetad\frac{f_\phi}{2} \right)\\
&+r^{-\sfrac{1}{2}}\left(r^{-1} f_{tt} -2r^{-1}\varthetad f_{t\phi} - r^{-1}\varthetadd f_\phi 
+r^{-1}\varthetad^2 f_{\phi\phi}  \right)\\
= &r^{-\sfrac32}\left(\frac f4- f_t+\varthetad f_\phi+ f_{tt} -2\varthetad f_{t\phi}-\varthetadd f_\phi 
+\varthetad^2 f_{\phi\phi} \right).
\end{align*}
In conclusion, we get
\begin{equation}\label{e:laplr}
f_t =\frac f4+ f_{\phi\phi}+ f_{tt}+\big(\varthetad f_\phi 
+\varthetad^2 f_{\phi\phi}-2\varthetad f_{t\phi}-\varthetadd f_\phi \big).
\end{equation}

Next, recalling equality \eqref {e:ftr}, we may  rewrite the Dirichlet condition in the new coordinates simply as 
\begin{equation}\label{e:dir f}
f (0, t)= f(2\pi,t)=0\, . 
\end{equation}

Finally, we derive the equation satisfied by the scalar curvature ${\bf k}$. To this end take into account that
\begin{align}\label{e:doteta}
\dot{\varrho} (t) =& - e^{-t} (\cos \vartheta (t), \sin \vartheta (t)) + e^{-t} \dot\vartheta (t) (-\sin \vartheta (t), \cos \vartheta (t))\nonumber\\
=& - \varrho (t) + \dot\vartheta (t) \varrho^\perp (t)\,, 
\end{align}
and thus differentiating \eqref{e:doteta} we get
\begin{align}\label{e:ddoteta}
\etadd (t) =& - \etad + \ddot\vartheta \varrho^\perp + \dot\vartheta \etad^\perp\, .
\end{align}
On the other hand, explicitely we have
\begin{align}
\dot{\varrho} (t)^\perp &= - e^{-t} (-\sin \vartheta (t), \cos \vartheta (t)) - e^{-t} \dot\vartheta (t) (\cos \vartheta (t), \sin \vartheta (t))\nonumber\\
&= - \varrho^\perp (t) - \dot\vartheta (t) \varrho (t) 
\, .
\end{align}
Hence, we conclude
\begin{align}
{\bf k} =& \frac{1}{|\etad(t)|}\left(\frac{d}{dt}\frac{\etad(t)}{|\etad(t)|}\right)
\cdot\frac{\etad^\perp(t)}{|\etad(t)|}=\frac{\etadd(t)\cdot\etad^\perp(t)}{|\etad(t)|^3}\nonumber\\
=& \frac{(\varthetad+\varthetad^3-\varthetadd) |\varrho (t)|^2}{(1+\varthetad^2)^{\sfrac32}|\varrho (t)|^3}
= r^{-1}\,\frac{\varthetad+\varthetad^3-\varthetadd}{(1+\varthetad^2)^{\sfrac32}}.\label{e:formula_curvatura}
\end{align}
As
\[
|\nabla u|^2=|\nabla w|^2=(w_r)^2+r^{-2}(w_\phi)^2=
r^{-1}\left(\frac f2+\varthetad f_\phi - f_t \right)^2+r^{-1} f_\phi^2
\]
we get
\begin{equation}\label{e:curvr}
 \frac{\varthetad+\varthetad^3-\varthetadd}{(1+\varthetad^2)^{\sfrac32}}= -
 \left. \left[\left(\frac f2+\varthetad f_\phi - f_t \right)^2 
 + f_\phi^2\right]\right|_0^{2\pi}\, .
\end{equation}
Thus, by taking into account \eqref{e:dir f} and \eqref{e:curvr}
we conclude the third equation in \eqref{e:SIS}.

In terms of $\vartheta$ the bound of $\alpha$ in \eqref{e:Bonnet2bis} reads as
\[
 \sup_{t\in[\sigma,\infty[}
 |\vartheta^{(i)}(t)|\leq C_i\,\delta \qquad \mbox{for every 
 $i\leq k$.}
\]
Indeed, differentiating $i$ times the identity $\vartheta (t) = \alpha (e^{-t})$ we get
\[
\vartheta^{(i)} (t) = \sum_{j=1}^i b_{i,j} e^{-jt} \alpha^{(j)} (e^{-t})\, ,
\]
with $b_{i,j}\in \mathbb R$.

Instead, the bound \eqref{e:decay_f} is a consequence of the linearity and continuity of the harmonic conjugation operator, i.e. the circular Hilbert transform, together with the decay \eqref{e:Bonnet2}. Indeed, the latter translates into 
\[
\sup_\phi|\partial_\phi^i\partial_t^j (g - \rad)|\leq C_i\,\delta\, 
\qquad\mbox{for every $t\in[\sigma,\infty[$ and $i+k\leq k$,}
\]
having set $g (\phi, t):= e^{\sfrac{t}2} u \big(\phi+\vartheta(t), e^{-t}\big)$.
Therefore, using the $\sfrac{1}{2}$-homogeneity of $\Rad$, we infer
\begin{align}\label{e:g}
g(\phi, t) - \rad (\phi) &=\;  e^{\sfrac{t}{2}} \left(u (\phi + \alpha (e^{-t}), e^{-t}) - \Rad (\phi, e^{-t})\right)
\notag\\
&=:\,  e^{\sfrac{t}{2}} h (\phi, e^{-t})\, .
\end{align}
We conclude that \eqref{e:Bonnet2} can be reformulated as
\[
\sup_{r\in]0,\sfrac12[}r^{i-\sfrac12}\|\partial^j_\theta \partial^i_r h (\cdot, r)\|_{C^0} \leq C_i\,\delta.
\]
On the other hand, differentiating \eqref{e:g} yields
\[
\partial_\phi^j \partial_t^i (g(\phi, t) - \rad (\phi)) = 
\sum_{\ell=0}^i b_{i, \ell}\, e^{\sfrac{t}{2} - \ell t} [\partial_\phi^j \partial_t^\ell h] (\phi, e^{-t})\,, 
\]
for some $b_{i,\ell}\in \mathbb R$. Setting $r= e^{-t}$, we then conclude 
\[
\|\partial_\phi^i\partial_t^j (g - \rad)\|_{C^0([0,2\pi]\times[\sigma,\infty[)} \leq \delta\,
\qquad\mbox{for all 
$i+j\leq k$}\,,
\]
and thus \eqref{e:decay_f} follows at once.
\end{proof}

\section{First linearization}

In this section we consider a sequence $(u_j, \alpha_j)$ as in Theorem \ref{t:main2} where condition \eqref{e:piccolezza} holds for a vanishing sequence $\e_0 (j)\downarrow 0$. Without loss of generality we assume $\alpha_j (e^{-a}) =0$ for some $a\geq 0$. We next define $\vartheta_j, \varrho_j$ and $f_j$ as in \eqref{e:vartheta}-\eqref{e:f}. Furthermore we fix $T_0>0$ and define:
\begin{align}
f^o_j (\phi, t) &:= {\textstyle{\frac{1}{2}}} (f_j (\phi, t) - f_j (2\pi - \phi, t)) \label{e:v_j}\\
\delta_j &:= \|f^o_j\|_{H^2 ([0,2\pi]\times [a,a+T_0])} + \|\dot{\vartheta_j}\|_{H^1 ([a,a+T_0])}\label{e:delta}\\
v_j (\phi, t) &:= \delta_j^{-1} f^o_j (\phi, t)\label{e:w}\\
\lambda_j (t) &:= \delta_j^{-1} \vartheta_j (t)\label{e:lambda}
\end{align}

\begin{remark}\label{r:even-odd-split}
It is moreover convenient to introduce the following terminology: a function $h$ on $[0,2\pi]\times [a,b]$ will be called even if $h (\phi, t) = h (2\pi-\phi, t)$ and odd if $h (\phi, t) = - h (2\pi - \phi, t)$. Moreover, a general $h$ can be split into the sum of its odd part $h_j^o (\phi, t) := \frac{h (\phi, t) - h (2\pi-\phi, t)}{2}$ and its even part $h_j^e (\phi, t) = \frac{h (\phi, t) + h (2\pi-\phi, t)}{2}$. Note finally that, if $h$ is even (resp. odd), then $\partial_\phi^j \partial_t^k h$ is odd (resp. even) for $j$ odd and even (resp. odd) for $j$ even. 
\end{remark}
Next, we show that the limit of $(v_j,\lambda_j)$ solves a linearization of \eqref{e:SIS}. 
In addition, for future purposes it is also necessary to take into account the linearization 
of \eqref{e:int-variation-bdry} (actually it suffices to consider \eqref{e:AM-identity}).
\begin{proposition}\label{p:linearizzazione}
Let $(u_j, \alpha_j)$ as in Theorem \ref{t:main2} where condition \eqref{e:piccolezza} holds for a vanishing sequence $\e_0 (j)\downarrow 0$. Assume $\alpha_j (e^{-a}) =0$ and define
$\vartheta_j, \varrho_j$ and $f_j$ as in \eqref{e:vartheta}-\eqref{e:f} and $v_j$ and $\lambda_j$ as above. Then, up to subsequences,
\begin{itemize}
\item[(a)] $v_j$ converges weakly in $H^2 ([0, 2\pi]\times [a,a+T_0])$ and uniformly to some odd $v$;
\item[(b)] $\lambda_j$ converges uniformly to some $\lambda$ in $[a,a+T_0]$;
\item[(c)]  the convergences are, respectively, in $C^{2,\alpha} ([0,2\pi]\times [a+\sigma, a+T_0-\sigma])$ 
for $v_j$ and in $C^{2,\alpha}([a+\sigma,a+T_0-\sigma])$ for $\lambda_j$ for all $\sigma\in (0, T_0/2), \alpha \in (0,1)$.
\end{itemize} 
Moreover, the pair $(v, \lambda)$ solves the following linear system of PDEs in $[0,2\pi]\times [a, a+T_0]$
\begin{equation}\label{e:lineare}
\left\{
\begin{array}{l}
v_t - v_{tt} = \frac{v}{4} + v_{\phi\phi} + (\dot\lambda - \ddot\lambda) \isq_\phi\\ \\
v (0, t) = v(2\pi, t) = 0\\ \\
\dot\lambda (t) - \ddot\lambda (t) = 2 \sqrt{{\textstyle{\frac{2}{\pi}}}} v_\phi (0,t)\\ \\
\lambda (0) = 0\, .
\end{array}\right.
\end{equation}
and satisfies the following integral condition for every $\sigma\in (a, a+T_0)$:
\begin{equation}\label{e:condizione_extra}
\int_0^{2\pi} \left[\left({\textstyle{\frac{v}{2}}} - v_t\right) (\phi, \sigma) \left(\cos {\textstyle{\frac{3\phi}{2}}} + \cos {\textstyle{\frac{\phi}{2}}}\right) + v_\phi (\phi, \sigma) \left(\sin {\textstyle{\frac{3\phi}{2}}} + \sin {\textstyle{\frac{\phi}{2}}}\right)\right]\, d\phi + \sqrt{{\textstyle{\frac{\pi}{2}}}} \dot\lambda (\sigma)=0\, .
\end{equation}
\end{proposition}

\begin{proof}[Proof of Proposition \ref{p:linearizzazione}] First of all, by a simple rescaling argument we can assume $a=0$. 
(a) and (b) are obvious consequences of the bounds on $(\lambda_j, v_j)$ (and of the fact that {$H^2 ([0,2\pi]\times[0,T_0])$ (resp. $H^2 ([0,T_0])$) embeds compactly in 
$C ([0,2\pi]\times[0,T_0])$ (resp. $C([0, T_0])$)}. Observe that, by assumption,
$\lambda_j (0) =0$ and thus $\lambda (0) =0$ is a consequence of the uniform convergence. Likewise the boundary condition $v (0, \cdot) = v (2\pi, 0) =0$ is also a consequence of uniform convergence and $v_j (0, \cdot) = v_j (0, \cdot) = 0$. 

We next observe that the PDE in \eqref{e:SIS} is linear in the unknown $f$. Hence, setting $f_j = f_j^e + \delta_j v_j$, we can take the odd part of each sides of the equation and, using Remark \ref{r:even-odd-split} infer
\begin{equation}\label{e:odd-PDE}
v_{j,tt} + v_{j,\phi\phi} = -\frac{v_j}{4} + (v_j)_t + (\ddot\lambda_j - \dot\lambda_j) f^e_{j,\phi} + 2 \dot\lambda_j f^e_{j, t\phi}  - \delta_j^2 \dot\lambda^2_j v_{j,\phi\phi}
\end{equation}
Observe that $f_j^e \to \isq$ in $C^2 ([0,2\pi]\times [0, T])$ and in $C^k ([0,2\pi]\times [\sigma, T-\sigma])$ for every $k$ and every $\sigma>0$ by Lemma \ref{l:nonlinear}. Passing into the limit we therefore conclude easily that $v$ solves the PDE in the first line of \eqref{e:lineare}.

We now rewrite the equation above in the following way:
\begin{align}
(1+ \delta_j^2 \dot \lambda_j^2) v_{j,tt} + v_{j,\phi\phi} = \underbrace{-\frac{v_j}{4} + v_{j,t} + (\ddot\lambda_j - \dot\lambda_j) f^e_{j,\phi} + 2 \dot\lambda_j f_{j,t\phi}}_{=:G_j}
\label{e:bootstrap}
\end{align}
Observe that, by our assumptions, the left hand side is an elliptic operator with a uniform bound on the ellipticity constants and a uniform bound on the $C^{1/2}$ norm of the coefficients. 

We next write the third equation in \eqref{e:SIS} in terms of $\lambda_j, v_j$ and $f^e_j$:
\begin{align}\label{e:transmission}
\dot\lambda_j - \ddot\lambda_j = \delta_j^2 \dot\lambda_j^3 + 4 (1+\delta_j^2 \dot\lambda_j^2)^{\sfrac{5}{2}} f^e_{j,\phi} (0,t) v_{j,\phi} (0,t)\, . 
\end{align}
Observe that, by the trace theorems, $v_{j,\phi} (0, \cdot)$ enjoys a uniform bound in $H^{1/2}$. Clearly, by the $C^2$ convergence of $f^e_j$ to $\isq$ we get that the third equation in \eqref{e:lineare} holds. Moreover, by the Sobolev embedding we conclude that the right hand side has a uniform control in $L^q$ for every $q<\infty$, in particular the same bound is enjoyed by $\ddot\lambda_j - \dot\lambda_j$ and, using that $\|\dot \lambda_j\|_{C^0}$ is bounded, we conclude that $\dot\lambda_j$ has a uniform $W^{1,q}$ bound for every $q<\infty$. 

Inserting the new estimates in \eqref{e:bootstrap} we can get a uniform bound on $\|G_j\|_{L^q ([0,2\pi]\times [0, T])}$ for every $q<\infty$. Using elliptic regularity we conclude a uniform bound for $\|v_j\|_{W^{2,q} ([0,2\pi]\times [\sigma, T-\sigma])}$. We now can use Morrey's embedding to get a uniform estimate on $\|v_j\|_{C^{1,\alpha} ([0,2\pi]\times [2\sigma, T-2\sigma])}$ for every $\alpha <1$. We now turn again to \eqref{e:transmission}, to conclude that the right hand side has a uniform $C^\alpha$ bound in $[2\sigma, T-2\sigma]$ for every $\alpha >0$. This gives uniform $C^{1,\alpha}$ bounds on the coefficient of the elliptic operator in the left hand side of \eqref{e:bootstrap} and uniform $C^\alpha$ bounds on the right hand side of  \eqref{e:bootstrap}. We can thus infer a uniform $C^{2,\alpha}$ bound in $[0,2\pi]\times [3\sigma, T-3\sigma]$ on $v_j$ from elliptic regularity.

It thus remains to prove \eqref{e:condizione_extra}. The latter will come from \eqref{e:AM-identity}. First of all, we fix $\sigma$, set $r_0 := e^{-\sigma}$ and observe that $\partial B_{r_0} \cap S_{u_j}$ consists of a single point $p_j$. We can thus apply Corollary \ref{c:AM-variations}. Hence using the relation between harmonic conjugates, we rewrite \eqref{e:AM-identity} as 
\[
\int_{\partial B_{r_0}\setminus \{p_j\}} \left(|\nabla w_j (q)|^2 \nu (q) \cdot \tau (p_j) - 2 \frac{\partial w_j}{\partial \tau}  (q)\nabla w_j (q) \cdot (\nu (q) - \nu (p_j))\right)\, d\mathcal{H}^1 (q) = 0\, . 
\]
Next, we assume without loss of generality that $p_j =0$ and rewrite the latter equality using polar coordinates:
\begin{equation}\label{e:A+B=0}
\underbrace{\int_0^{2\pi} \Big(r_0 w_{j,r}^2 -  \frac{1}{r_0} w_{j,\phi}^2\Big) (\phi, r_0) \sin \phi\, r_0\, d\phi}_{=:A_j}
- \underbrace{2 \int_0^{2\pi} \left(w_{j,r} w_{j,\phi}\right) (\phi, r_0) (1 + \cos \phi)\, d\phi}_{=:B_j} = 0
\end{equation}
We next write $w_j$ in terms of $v_j$, $\gamma_j$ and $\lambda_j$  as
\[
w_j (r, \phi) = r^{\sfrac{1}{2}} f^e_j (\phi - \delta_j \lambda_j (-\ln r), - \ln r)
+ \delta_j r^{\sfrac{1}{2}} v_j (\phi - \delta_j \lambda_j (-\ln r), -\ln r)\, \, ,
\]
Note that, having normalized so that $p_j =0$, we conclude that $\lambda_j (-\ln r_0) = \lambda_j (\sigma) = 0$. Using the latter we compute:
\begin{align}
\frac{\partial w_j}{\partial r} (\phi, r_0) = & \underbrace{r_0^{-\sfrac{1}{2}} \left(\frac{f_j^e}{2} - f_{j,t}\right) (\phi, \sigma)}_{=:a_j (\phi)} + \delta_j \underbrace{r_0^{-\sfrac{1}{2}} \left(\frac{\dot \lambda_j (\sigma)}{\sqrt{2\pi}} \cos \frac{\phi}{2} +\frac{v_j (\phi, \sigma)}{2} - v_{j,t} (\phi, \sigma)\right)}_{=: b_j (\phi)} + o (\delta_j)\label{e:dr10}\\
\frac{\partial w_j}{\partial \phi} (\phi, r_0) = & \underbrace{r_0^{\sfrac{1}{2}} f^e_{j,\phi}}_{=: c_j (\phi)} + \delta_j \underbrace{r_0^{\sfrac{1}{2}} v_{j,\phi} (\phi, \sigma)}_{=:d_j (\phi)}\label{e:dphi10}\, .
\end{align}
Note now that the function $a_j^2 - \frac{1}{r_0} c_j^2$ is even. Since $\sin \phi$ is odd, we thus conclude
\begin{align*}
A_j &= 2 \delta_j \int_0^{2\pi} \left(r_0 a_j (\phi) b_j (\phi) - r_0^{-1} c_j (\phi) d_j (\phi)\right)\, \sin \phi\, d\phi\,  + o (\delta_j)\, . 
\end{align*}
Letting $j\to\infty$ we obtain
\begin{align}
\lim_{j\to\infty} \delta_j^{-1} A_j &= \sqrt{\frac{2}{\pi}} \int_0^{2\pi} \left(\sin \frac{\phi}{2} \left(\frac{\dot \lambda (\sigma)}{\sqrt{2\pi}} \cos \frac{\phi}{2} +\frac{v (\phi, \sigma)}{2} - v_t (\phi, \sigma)\right) - \cos \frac{\phi}{2} v_\phi (\phi, \sigma) \right)\sin \phi\, d\phi\, .\label{e:limite-A}
\end{align}
Similarly, $a_j c_j$ is odd and therefore $\int a_j c_j (1+ \cos \phi)\, d\phi = 0$, from which we conclude
\begin{align*}
B_j & = 2 \delta_j \int_0^{2\pi} (a_j (\phi) d_j (\phi) + b_j (\phi) c_j (\phi)) (1+ \cos \phi)\, d\phi + o (\delta_j)\, .  
\end{align*}
Hence 
\begin{align}
& \lim_{j\to \infty} \delta_j^{-1} B_j\nonumber\\ 
= & \sqrt{\frac{2}{\pi}} \int_0^{2\pi} \left(\sin \frac{\phi}{2} v_\phi (\phi, \sigma) + \cos \frac{\phi}{2}\left(\frac{\dot \lambda (\sigma)}{\sqrt{2\pi}} \cos \frac{\phi}{2} +\frac{v (\phi, \sigma)}{2} - v_t (\phi, \sigma)\right)\right) (1+\cos \phi)\, d\phi\, .\label{e:limite-B}
\end{align}
Combining \eqref{e:limite-A} and \eqref{e:limite-B} with \eqref{e:A+B=0} we conclude
\begin{align*}
0 = & \frac{\dot\lambda (\sigma)}{2\pi} \int_0^{2\pi} (\sin^2 \phi - (1+\cos \phi)^2)\, d\phi
 - \sqrt{\frac{2}{\pi}} \int_0^{2\pi} v_\phi (\phi, \sigma) \left( \sin \frac{\phi}{2} + \sin \frac{3\phi}{2} \right)\, d\phi\\
 & -\sqrt{\frac{2}{\pi}} \int_0^{2\pi} \left( \frac{v}{2} (\phi, \sigma) - v_t (\phi, \sigma)\right) \left(\cos \frac{\phi}{2} + \cos \frac{3\phi}{2}\right)\, d\phi\, .
\end{align*}
Using $\int_0^{2\pi} (\sin^2 \phi - (1+\cos \phi)^2) d\phi = -2\pi$, we conclude \eqref{e:condizione_extra}.
\end{proof}

\section{Spectral analysis}

In this section we will find a suitable representation of odd solutions of \eqref{e:lineare} on a domain $[0,2\pi]\times [0, T]$, based on the spectral analysis of a closely related linear PDE. 
Next we consider the change of variables
\begin{equation}\label{e:zeta}
\zeta (\phi, t) := v (\phi, t) - \lambda (t) \isq_\phi (t) = v (\phi, t) - \frac{\lambda (t)}{\sqrt{2\pi}} \cos \frac{\phi}{2}\, .
\end{equation}

\begin{lemma}\label{l:Ventsell}
The pair $(v, \lambda)\in H^2\times H^3$ solves \eqref{e:lineare} if and only if $\zeta(\cdot,t)$ is odd, $\zeta (0,0) = 0$
$\zeta(0,t)=-\frac{\lambda(t)}{\sqrt{2\pi}}$, and $\zeta$ solves
the following partial differential equation with Ventsel boundary conditions:
\begin{equation}\label{e:Ventsell}
\left\{
\begin{array}{l}
\zeta_{tt} + \zeta_{\phi\phi} + \frac{\zeta}{4} - \zeta_t = 0\\ \\
\zeta_\phi (0, t) + \frac{\pi}{2} \left(\frac{\zeta}{4} (0,t) + \zeta_{\phi\phi} (0,t)\right) = 0\, .
\end{array}
\right.
\end{equation} 
\end{lemma}

Taking into account standard regularity theory, the lemma is reduced to elementary computations which are left to the reader.
We aim at a representation for $\zeta$, i.e. a representation as a series of functions in $\phi$ with coefficients depending on $t$ for which we can reduce \eqref{e:Ventsell} to an independent system of ODEs for the coefficients. To that aim we introduce the space 
\begin{equation}\label{e:O}
\mathcal{O} := \{g \in H^1 (]0, 2\pi[): g (\phi) = - g (2\pi -\phi)\}\, ,
\end{equation}
The representation is detailed in the following 

\begin{proposition}\label{p:odd}
\begin{equation}\label{e:representation_odd}
\zeta(\phi,t)= \sum_{k=0}^\infty a_k(t)\zeta_k (\phi)\, ,
\end{equation}
where:
\begin{itemize}
\item[(a)] $C^{-1} \sum_k a_k^2 (t) \leq \|\zeta (\cdot, t)\|^2_{H^1}\leq C \sum_k a_k^2 (t)$ for a universal constant $C$; 
\item[(b)] The functions $\zeta_k$ are defined in Section \ref{s:spectral};
\item[(c)] For $k\geq 2$, the coefficients $a_k$ satisfy $a_k (t) = \langle \zeta (\cdot, t), \zeta_k \rangle$ for the bilinear symmetric form $\langle\cdot, \cdot\rangle$ defined in Section \ref{s:ventsel} (cf. \eqref{e:quasi_prodotto}), while the 
coefficients $a_0 (t)$ and $a_1 (t)$ are given by $a_0 (t) = \mathcal{L}_0 (\zeta (\cdot, t))$ and $a_1 (t) = \mathcal{L}_1 (\zeta (\cdot, t))$ for appropriately defined linear bounded functionals $\mathcal{L}_0, \mathcal{L}_1 : \mathcal{O} \to \mathbb R$. 
\end{itemize}
Next, if $\zeta\in H^2 ([0,2\pi]\times [0, T])$ is odd and solves \eqref{e:Ventsell} then $\zeta\in C^\infty ([0,2\pi]\times (0, T))$ and
for every $k\geq 2$ the coefficients $a_k (t)$ in the expansion satisfy $a_k''(t) - a_k' (t) =  (\nu_k^2 - \frac{1}{4}) a_k (t)$, where the number $\nu_k$'s are given in Lemma \ref{l:autovalori}.
\end{proposition}

The proof is an obvious consequence of Proposition \ref{p:spettro}, which will be the main focus of this section.

\subsection{The Ventsel boundary condition}\label{s:ventsel}
For every $g\in \mathcal{O}$ we look for solutions $h\in \mathcal{O}$ of the following equation:
\begin{equation}\label{e:ventsel}
 \begin{cases}
 \displaystyle{h_{\phi\phi} = g}\cr\cr
 \displaystyle{h_\phi (0) = - \frac{\pi}{2} \left(\frac{h (0)}{4} + h_{\phi\phi} (0)\right) \, .}
\end{cases}
\end{equation}
The following is an elementary fact of which we include the proof for the reader's convenience.
\begin{lemma}\label{l:risolvente}
For every $g \in \mathcal{O}$ there is a unique solution $h:= \mathscr{A} (g)\in \mathcal{O}$ of 
\eqref{e:ventsel}. In fact the operator $\mathscr{A} \colon \mathcal{O} \to \mathcal{O}$ is compact.
\end{lemma}
\begin{proof}
$h\in \mathcal{O}$ solves the first equation in \eqref{e:ventsel} if and only if
\begin{equation}\label{e:sol_esplicita}
h (\phi) = h_\phi (\pi) (\phi -\pi) + \underbrace{\int_\pi^\phi \int_\pi^\tau g (s)\, ds\, d\tau}_{=: G (\phi)}\, .
\end{equation}
On the other hand the initial condition holds if and only if 
\[
h_\phi(\pi) \left(\frac{\pi^2}{8}-1\right) =  G' (0) + \frac{\pi}{2} \left(\frac{G(0)}{4} + G'' (0)\right)\, .
\]
Since $G$ is determined by $g$, the latter determines uniquely $h_\phi(\pi)$ and thus shows that there is one
and only one solution $h = \mathscr{A} (g) \in \mathcal{O}$ of \eqref{e:ventsel}. Moreover, we obviously have
\[
\|\mathscr{A} (g)\|_{H^3} \leq C \|g\|_{H^1}\, ,
\]
which shows that the operator is compact.
\end{proof}

We next introduce in $\mathcal{O}$ a continuous bilinear map
\begin{equation}\label{e:quasi_prodotto}
\langle u, v \rangle := \int_0^{2\pi} u_\phi v_\phi - \frac{1}{4} \int_0^{2\pi} uv\, .
\end{equation}
If $\langle \cdot, \cdot \rangle$ were a scalar product on $\mathcal{O}$, $\mathscr{A}$ would be a 
self-adjoint operator on $\mathcal{O}$ with respect to it and we would conclude that there is an 
orthonormal base made by eigenfunctions of $\mathscr{A}$. Unfortunately $\langle \cdot, \cdot \rangle$ is 
{\em not} positive definite. This causes some technical complications.

\begin{lemma}\label{l:autoaggiunto} 
The bilinear map $\langle \cdot, \cdot\rangle$ satisfies the following properties:
\begin{itemize}
\item[(a)] $\langle v, v \rangle \geq 0$ for every $v\in \mathcal{O}$;
\item[(b)] $\langle v, v\rangle = 0$ if and only if $v (\phi) = \mu \cos \frac{\phi}{2}$ for some constant $\mu$;
\item[(c)] $\langle v, \cos \frac{\phi}{2}\rangle = 0$ for every $v\in \mathcal{O}$;
\item[(d)] $\langle \mathscr{A} (v), w\rangle = \langle v, \mathscr{A} (w)\rangle$ for every $v,w\in \mathcal{O}$.
\end{itemize}
\end{lemma}
\begin{proof} {\bf (a) \& (b)}
First observe that (a) is equivalent to
\begin{equation}\label{e:sharp}
\frac{1}{4} \int_0^{2\pi} v^2 \leq \int_0^{2\pi} v_\phi^2\, .
\end{equation}
If we write $v$ using the Fourier series expansion
$v (\phi) = \sum_{k=1}^\infty \alpha_k \cos \frac{k\phi}{2}$,
the inequality becomes obvious and
it is also clear that equality holds if and only if $\alpha_k =0$ for every $k\geq 2$. 

\medskip

{\bf (c)} Let $z (\phi) := \cos \frac{\phi}{2}$ and observe that $\frac{z}{4} + z_{\phi\phi}=0$ and that
$z_\phi (0)= z_\phi (2\pi) = 0$. We therefore compute
\begin{align*}
\langle w,z\rangle = & \int_0^{2\pi} w_\phi z_\phi - \frac{1}{4} \int_0^{2\pi} zw =  w z_\phi \Big|_0^{2\pi} - \int_0^{2\pi} w \left(z_{\phi\phi} + \frac{z}{4}\right) = 0\, .
\end{align*}

\medskip

{\bf (d)} Consider $z = \mathscr{A} (v)$ and $u = \mathscr{A} (w)$. We then compute
\begin{align*}
\langle \mathscr{A} (v), w \rangle = & \langle z, u_{\phi\phi}\rangle = \int_0^{2\pi} z_\phi u_{\phi\phi\phi} - \frac{1}{4} \int_0^{2\pi} z u_{\phi\phi}\\
=& z_\phi u_{\phi\phi} \Big|_0^{2\pi} - \int_0^{2\pi} z_{\phi\phi} u_{\phi\phi} - \frac{1}{4} z u_\phi \Big|_0^{2\pi} +
\frac{1}{4} \int_0^{2\pi} z_\phi u_\phi\\
=& z_\phi u_{\phi\phi} \Big|_0^{2\pi} - z_{\phi\phi} u_{\phi}\Big|_0^{2\pi} + \int_0^{2\pi} z_{\phi\phi\phi} u_\phi  - \frac{1}{4} z u_\phi \Big|_0^{2\pi}
+  \frac{1}{4} z_\phi u \Big|_0^{2\pi}- \frac{1}{4} \int_0^{2\pi} z_{\phi\phi} u\\
= & \left. z_\phi \left(u_{\phi\phi} + \frac{u}{4}\right)\right|_0^{2\pi} - \left. u_\phi \left(z_{\phi\phi} + \frac{z}{4}\right)\right|_0^{2\pi} + \langle z_{\phi\phi} , u\rangle\\
= & - \frac{2}{\pi} z_\phi u_\phi\Big|_0^{2\pi} + \frac{2}{\pi} z_\phi u_\phi \Big|_0^{2\pi} + \langle v, \mathscr{A} (w)\rangle =
\langle v , \mathscr{A} (w)\rangle\, . \qedhere
\end{align*}
\end{proof}

\subsection{Spectral decomposition}\label{s:spectral} We are now ready to prove the following spectral analysis. First of all we start with the following

\begin{lemma}\label{l:autovalori} If $\mu$ is a real number and $h\in \mathcal{O}$ a solution of the following eigenvalue problem
\begin{equation}\label{e:autovalore}
\left\{
\begin{array}{l}
h_{\phi\phi} = \mu h\\ \\
h_\phi (0) = - \frac{\pi}{2} \left(\frac{h (0)}{4}+ h_{\phi\phi} (0) \right)
\end{array}\right.
\end{equation}
then
\begin{itemize}
\item[(a)] $\mu < 0$ and if we set $\mu = -\nu^2$ for $\nu>0$, then $\nu$ is a positive solution of
\begin{equation}\label{e:zeros}
\nu \cos \nu \pi = \frac{\pi}{2} \left(\frac{1}{4} - \nu^2\right) \sin \nu \pi\, .
\end{equation}
\item[(b)] $h$ is a constant multiple of $\sin (\nu (\phi -\pi))$. 
\item[(c)] The positive solutions of \eqref{e:zeros} are given by an increasing sequence $\{\nu_k\}_k\in \mathbb N$ in which $\nu_1 = \frac{1}{2}$, $\nu_2>\frac{3}{2}$ and 
\begin{equation}\label{e:asintotico}
\lim_{k\to \infty} \frac{\nu_k}{k} = 1\, 
\end{equation}
\end{itemize}
\end{lemma}

We will postpone the proof of the lemma and introduce instead the following notation. For $k=1$ we set $\zeta_1 (\phi) = \cos \frac{\phi}{2}$, while for $k>1$ we let $\zeta_k := c_k 
\sin(\nu_k(\phi-\pi))$, where $c_k$ is chosen so that $\langle \zeta_k, \zeta_k\rangle = 1$.  Furthermore we set $\zeta_0 (\phi) := (\phi-\pi) \sin \frac{\phi}{2}$, the relevance of the latter function is that it solves 
\begin{equation}\label{e:blocchetto}
\left\{
\begin{array}{l}
\zeta_{\phi\phi} = - \frac{\zeta}{4} + \zeta_1\\ \\
\zeta_\phi (0) = - \frac{\pi}{2} \left(\frac{\zeta (0)}{4} + \zeta_{\phi\phi} (0)\right)\, .
\end{array}\right.
\end{equation}
In particular if we restrict the second derivative operator on the $2$-dimensional vector space generated by $\zeta_1$ and $\zeta_0$, its matrix representation is given by
\[
\left(
\begin{array}{ll}
-4 & 0\\
1 & -4\, .
\end{array}
\right)
\]
Consequently the operator $\mathscr{A}$ is not diagonalizable in $\mathcal{O}$, which is the reason why its spectral analysis is somewhat complicated.

\begin{proposition}\label{p:spettro}
The set $\{\zeta_k \} \subset \mathcal{O}$ is an Hilbert basis for $\mathcal{O}$, namely for every $\zeta \in \mathcal{O}$ there is a unique 
choice of coefficients $\{a_k\}$ such that
\begin{equation}\label{e:espansione}
\zeta = \sum_{k=0}^\infty a_k \zeta_k \, ,
\end{equation}
where the series converges in $H^1$. The coefficients $a_k$ in \eqref{e:espansione} are determined by
\begin{equation}\label{e:coefficienti}
a _k = \langle \zeta, \zeta_k\rangle\qquad \mbox{for all $k\geq 2$,}
\end{equation}
while $a_0$ and $a_1$ are continuous linear functionals on $\mathcal{O}$. 
\end{proposition} 

\begin{proof}[Proof of Lemma \ref{l:autovalori}]
First of all, consider $\mu =0$. An odd solution of \eqref{e:autovalore} must then take necessarily the form $c (\phi - \pi)$ and the boundary condition would imply $c=0$. If $\mu >0$ observe that a nontrivial function $\zeta\in \mathcal{O}$ solving \eqref{e:autovalore} would also satisfy $\mathscr{A} (\zeta) = \frac{\zeta}{\mu}$. If $ \mu = \nu^2 >0$ for $\nu>0$, then $h(\phi ) = c (e^{\nu (\phi-\pi)} - e^{-\nu (\phi-\pi)})$ for some constant $c$. If $c\neq 0$ the boundary condition becomes
\begin{equation}
\nu \left(e^{-\nu \pi} + e^{\nu \pi}\right) = - \frac{\pi}{2} \left(\frac{1}{4} + \nu^2\right) \left(e^{-\nu \pi} - e^{\nu \pi}\right)\, .
\end{equation}
The latter identity is equivalent to
\begin{equation}
e^{2\pi \nu} ({\pi + 4\pi \nu^2 - 8 \nu})= {\pi + 4\pi \nu^2 + 8\nu}\, .
\end{equation}
If we make the substitution $x= 2\pi \nu$, we then are looking for zeros of the function
\[
\Phi (x) = e^x (\pi^2 + x^2 - 4x) - \pi^2 - x^2 - 4x = 0\, .
\]
The derivative is given by 
\[
\Phi' (x) = e^x (x^2 -2x + \pi^2 -4) - 2 (2+x)\, ,
\]
the second derivative by
\[
\Phi'' (x) = e^x (x^2 + \pi^2 -6) -2 \geq 3 e^x -2 > 0\, .
\]
In particular $\Phi$ is convex and $\Phi' (0) = \pi^2 - 8 > 0$. Thus $\Phi$ is strictly increasing and, since $\Phi (0) = 0$, it cannot have positive zeros.

Consider now $\mu = - \nu^2$ for $\nu>0$. A solution of the PDE in \eqref{e:autovalore} must then be a linear combination of $\sin \nu (\phi-\pi)$ and 
$\cos \nu (\phi-\pi)$: the requirement that $h\in \mathcal{O}$ excludes the multiples of $\cos \nu (\phi-\pi)$ in the linear combination. 

For $h (\phi) = \sin \nu (\phi-\pi)$ the 
boundary condition becomes 
\begin{equation}\label{e:autovalori}
\nu \cos (-\nu \pi) = - \frac{\pi}{2} \left(\frac{1}{4} - \nu^2\right) \sin (-\nu \pi)\, ,
\end{equation}
which is equivalent to \eqref{e:zeros}. If we introduce the unknown $x = \pi \nu$, then the equation becomes
\[
\Psi (x) := 8 x \cos x - \left(\pi^2-4 x^2\right) \sin x = 0\, .
\]
Since $\Psi' (x) = (4x^2 + 8 - \pi^2) \cos x$,
$\Psi'$ has a single zero in the open interval $]0, \frac{\pi}{2}[$. Since $\Psi (0) = \Psi (\frac{\pi}{2})=0$, 
we infer that there is no zero of $\Psi$ in the open interval $]0, \frac{\pi}{2}[$, i.e. any positive $\nu$ 
satisfying \eqref{e:autovalori} cannot be smaller than $\frac{1}{2}$. Moreover, as $\Psi'$ is strictly negative on 
$]\frac\pi 2,\frac32\pi[$ and $\Psi(\frac32\pi)<0<\Psi(2\pi)$, the next solution $\nu$ lies in 
$]\frac 32,2[$. 

Next, there is a unique solution $\nu_k\in]k-1,k[$, for every $k\geq 3$.
Indeed, $\Psi((k-1)\,\pi)\cdot\Psi(k\,\pi)<0$ and $\Psi'$ has a single zero in the open interval $](k-1)\,\pi, k\pi[$.
Therefore $(\nu_k)_k$ satisfies \eqref{e:asintotico}.
\end{proof}

\begin{proof}[Proof of Proposition \ref{p:spettro}]
Let $Y$ be the closure in $H^1$ of the vector space $V$ generated by $\{\zeta_k\}_{k\geq 2}$. First of all observe that, for some constant $C$ independent of $k$, 
\begin{equation}\label{e:comparable_1}
1 = \langle \zeta_k, \zeta_k \rangle \geq C^{-1} \|\zeta_k\|_{H^1}^2\qquad \forall k\geq 2\, .
\end{equation}
Indeed set $g_k:= \sin \nu_k (\phi-\pi)$: \eqref{e:comparable_1} is then equivalent to say that the $g_k$'s satisfy the same inequality. An explicit computation shows that this is equivalent to
\begin{align*}
& \int_0^{2\pi} \cos^2 \nu_k (\phi - \pi)\, d\phi  - \frac{1}{4 \nu_k^2} \int_0^{2\pi} \sin^2 \nu_k (\phi-\pi)\, d\phi\\
\geq &\; C^{-1} \left(\int_0^{2\pi} \cos^2 \nu_k (\phi-\pi)\, d\phi +  \frac{1}{\nu_k^2} \int_0^{2\pi} \sin^2 \nu_k (\phi-\pi)\, d\phi\right)\, .
\end{align*}
For each fixed $\nu_k$ the fact that the inequality holds for a sufficiently large constant is an easy consequence of the fact that $\int \cos^2 \nu_k (\phi-\pi)$ is positive while 
$\int \sin^2 \nu_k (\phi-\pi)$ is finite. On the other hand by \eqref{e:asintotico} both integrals converge to $\pi$ as $k\uparrow \infty$ and thus for a sufficiently large $k$ the inequality holds for $C\geq 2$. Now, for $k\neq j$ we have 
\[
\langle \zeta_k, \zeta_j\rangle = - \nu_k^2 \langle \mathscr{A} (\zeta_k), \zeta_j\rangle = - \nu_k^2 \langle \zeta_k, \mathscr{A} (\zeta_j)\rangle 
= \frac{\nu_k^2}{\nu_j^2} \langle \zeta_k, \zeta_j\rangle
\]
implying that $\langle \zeta_k, \zeta_j \rangle =0$.

We next claim that $\zeta_1 (\phi) = \cos \frac{\phi}{2}\not\in Y$. Otherwise there is a sequence $\{v_n\}\subset V$ such that $v_n\to \zeta_1$ strongly in $H^1$. $v_n$ takes therefore the form $v_n = \sum_{k=2}^{N (n)} a_{n,k} \zeta_k$. Using that $\langle v_n, v_n \rangle$ converges to $\langle \zeta_1, \zeta_1\rangle = 0$. Thus we have
\begin{equation}\label{e:si_annulla}
\lim_{n\to \infty} \sum_{k=2}^{N(n)} a_{n, k}^2 = 0\, .
\end{equation}
Now, given that the operator $\mathscr{A}$ is compact we also have that $z_n := \frac{\mathscr{A} (v_n)}{4}$ converges strongly in $H^1$ to 
$\frac{\mathscr{A} (\zeta_1)}{4} = - \cos \frac{\phi}{2}$. On the other hand
\[
z_n = - \sum_{k=2}^{N(n)} \frac{1}{4\nu_k^2} a_{n, k} \zeta_k \, .
\]
We then would have by item (c) of Lemma \ref{l:autovalori} and \eqref{e:comparable_1}
\begin{align*}
0 < & \|\zeta_1\|^2_{H^1} = \lim_{n\to \infty} \|z_n\|_{H^1}^2 \leq \lim_{n\to \infty} \sum_{k, j=2}^{N(n)} \frac{|a_{n,j}| 
|a_{n,k}|}{16\nu_j^2 \nu_k^2} \|\zeta_k\|_{H^1}\|\zeta_j\|_{H^1}\\
\leq & C \limsup_{n\to\infty} \left(\sum_{k=2}^{N(n)} \frac{|a_{n,j}|}{j^2}\right)^2
\leq C \limsup_{n\to\infty} \sum_{k=2}^{N(n)} \frac{1}{k^4} \sum_{j=2}^{N(n)} a_{n,j}^2
\leq  C \limsup_{n\to\infty}\sum_{j=2}^{N(n)} a_{n,j}^2\stackrel{\eqref{e:si_annulla}}{=} 0\, ,
\end{align*}
Consider now the standard $H^1$ scalar product $(\cdot, \cdot)$ on $\mathcal{O}$ and for every $\zeta\in Y$ let 
$\zeta = \zeta^\perp + \zeta^\parallel$ be the decomposition of $\zeta$ into a multiple of $\zeta_1$ and an element $\zeta^\perp$ orthogonal in the scalar product $(\cdot, \cdot)$ to $\zeta_1$. Since $\zeta_1\not \in Y$ and $Y$ is closed in $H^1$, there is a constant $\alpha >0$ such that $\|\zeta^\perp\|^2_{H^1} \geq \alpha \|\zeta\|^2_{H^1}$. On the other hand using the Fourier expansion of $\zeta$ we easily see that $\langle \zeta, \zeta\rangle = \langle \zeta^\perp, \zeta^\perp\rangle \geq C^{-1} \|\zeta^\perp\|^2_{H^1}$ for some universal constant $C>0$. In particular $\mathscr{A}$ is a compact self-adjoint operator on $Y$, which implies that $\{\zeta_k\}_{k\geq 2}$ is an orthonormal basis on the Hilbert space $Y$ (endowed with the scalar product $\langle \cdot, \cdot \rangle$).

Consider now the $2$-dimensional vector space 
$Z := \{a_0 \zeta_0 +a_1 \zeta_1 : a_i \in \mathbb R\}$. If $a_0 \zeta_0 + a_1\zeta _1 = z\in Z \cap Y$, using Lemma \ref{l:autoaggiunto} and the fact that $\langle y, \zeta_1\rangle =0$ for every $y\in Y$, we can compute
\[
\langle z, \zeta_j\rangle = a_0 \langle \zeta_0, \zeta_j\rangle = - \nu_j^2 \langle a_0 \zeta_0, \mathscr{A} (\zeta_j)\rangle = -\nu_j^2 \langle a_0 \mathscr{A} (\zeta_0), \zeta_j\rangle
= 4\nu_j^2 \langle a_0 \zeta_0, \zeta_j\rangle = 4 \nu_j^2 \langle z, \zeta_j\rangle\, 
\] 
for every $j\geq 2$. Since $\nu_j> \frac{3}{2}$ we infer that $\langle z, \zeta_j\rangle =0$, i.e. that $z=0$, since $\{\zeta_j\}_{j\geq 2}$ is an orthonormal Hilbert basis of $Y$ with respect to the scalar product $\langle \cdot, \cdot\rangle$. We have thus concluded that $Z\cap Y = \{0\}$. The proof of the proposition will be completed once we show that $Z+Y = \mathcal{O}$. Consider an element $\zeta\in \mathcal{O}$ and define 
\[
\bar\zeta := \frac{\langle \zeta_0, \zeta\rangle}{\langle \zeta_0, \zeta_0\rangle} \zeta_0 + \sum_{j\geq 2} \langle \zeta_j, \zeta\rangle \zeta_j\, .
\]
It turns out that $\bar\zeta \in Z+Y$ and that $\hat{\zeta} := \zeta - \bar\zeta$ satisfies the condition $\langle \hat\zeta, z\rangle = 0$ for every element $z\in Z+Y=:X$. We claim that the latter condition implies that $\hat\zeta$ is a constant multiple of $\cos \frac{\phi}{2}$. Indeed set $X^\perp := \{v : \langle v, w\rangle = 0 \quad \forall w\in X\}$. Then clearly $\mathscr{A} (X^\perp) \subset X^\perp$. Moreover $\mathscr{A}$ on $X^\perp$ has only one eigenvalue, namely $-4$. Consider now $X^\perp \ni v \mapsto Q (v,v) = \langle \mathscr{A} (v), \mathscr{A} (v)\rangle = \langle \mathscr{A}^2 (v), v \rangle$ and set
\begin{equation}\label{e:massimo}
m:= \sup \{Q (v,v): v\in  X^\perp \quad \mbox{and}\quad \langle v, v\rangle =1\}\, ,
\end{equation}
where at the moment $m$ is allowed to be $\infty$ as well. If $m=0$ we then have that $\mathscr{A} (v)$ is a multiple of $\zeta_1$ for every $v$ and this would imply that $v$ itself is a multiple of $\zeta_1$. We therefore assume that $m$ is nonzero.
Using the fact that $Q (v, \zeta_1) = 0$ for every $v$, we can find a maximizing sequence with Fourier expansion
\[
v_k := \sum_{j\geq 1} c_{k,j} \cos \frac{2j+1}{2} \phi  
\]
for which we easily see that $\langle v_k, v_k \rangle \geq C^{-1} \|v_k\|_{H^1}^2$. We can thus extract a subsequence converging weakly to some $v$. $v$ clearly belongs to $X^\perp$ and, by the compactness of the operator $\mathscr{A}$ is actually a maximizer of \eqref{e:massimo}. The Euler-Lagrange condition implies then that $\mathscr{A}^2 (v) = m v + b \zeta_1$ for some real coefficients $b$. Consider now the vector space $W$ generated by $\zeta_1, v$ and $\mathscr{A} (v)$. $W$ is then either $2$-dimensional or $3$-dimensional and $\mathscr{A}$ maps it onto itself. If $W$ were three-dimensional, then the matrix representation of $\mathscr{A}|_W$ in the basis  $\zeta_1, v$ and $\mathscr{A} (v)$ would be
\[
\left(
\begin{array}{lll}
-4 & 0 & 0\\
0 & 1 & 0\\
\alpha & 0 & m\\
\end{array}
\right)
\]
Since the characteristic polynomial of the latter matrix is $(x-1) (x-m)(x+4)$, $\mathscr{A}$ would have an eigenvalue different from $-4$ on $W\subset X^\perp$, which is not possible. 
On the other hand if $W$ were $2$-dimensional, then $v$ and $\cos \frac{\phi}{2}$ would be a basis and the matrix representation of $\mathscr{A}|_W$ in that basis would be
\[
\left(
\begin{array}{ll}
-4 & 0\\
\alpha & \beta
\end{array}
\right)
\]
Since $\mathscr{A}|_W$ cannot have an eigenvalue different than $-4$ this would force $\beta = -4$. We then would have $\mathscr{A} (v) = -4v + \alpha \zeta_0$. This would imply that $v$ is an odd solution of $v_{\phi\phi} +\frac{v}{4} = \alpha \cos \frac{\phi}{2}$. The general solution of the latter equation is given by $c_1 \cos \frac{\phi}{2} + c_2 \sin \frac{\phi}{2} + \alpha (\phi - \pi) \sin \frac{\phi}{2}$, for real coefficients $c_1$ and $c_2$. The fact that $v$ is odd implies $c_2=0$, namely $c_1 \zeta_1 + \alpha \zeta_0$. The fact that $v$ is not colinear with $\zeta_1$ implies that $\alpha \neq 0$, but on the other hand since $v\in X^\perp$, $\langle v, \zeta_1\rangle =0$, which implies $\alpha =0$. We have reached a contradiction: $X^\perp$ was thus the line generated by $\zeta_1$, proving that indeed $X=\mathcal{O}$.  
\end{proof}

\section{The three annuli property}\label{s:tre-anelli}

We now define a functional which will be instrumental in proving a suitable decay property for coefficients of solutions of \eqref{e:lineare} and hence of \eqref{e:SIS}. 

\begin{definition}\label{d:functionals}
Fix a constant $c_0>0$ appropriately small (whose choice will be specified later). Consider now any $\sigma<s$ real numbers and a pair of functions $(v, \lambda)$ such that
\begin{itemize}
\item[(i)] $v$ is odd, $v\in H^2 ([0,2\pi]\times [\sigma,s])$ and $v (0, t)=v(2\pi,t)=0$ for every $t$;
\item[(ii)] $\lambda\in H^2 ([\sigma,s])$.
\end{itemize}
Define $\zeta$ as in \eqref{e:zeta} and let $a_k (t)$ be the coefficients in the representation \eqref{e:representation_odd} and $\nu_k$ the numbers in Lemma \ref{l:autovalori}. We then define the functionals
\begin{align}
\mathcal{E} (v, \lambda, \sigma,s) &:= \sum_{k\geq 2} \int_\sigma^s  ({\nu_k^4 a_k (t)^2 + a''_k (t)^2})\, dt\, \label{e:def_E}\\
\mathcal{F} (v, \lambda,\sigma,s) &:= \int_\sigma^s (\dot\lambda (t)^2 + \ddot\lambda (t)^2 + a_0 (t)^2 + a_1 (t)^2 + a_0'' (t)^2 + a_1''(t)^2)\, dt\\
\mathcal{G} (v, \lambda, \sigma,s) &:= \max\{\mathcal{E} (v,\lambda, \sigma,s), c_0 \mathcal{F} (v, \lambda,\sigma,s)\}
\end{align}
\end{definition}

\begin{proposition}\label{p:tre-anelli}
There is a constant $\eta>0$ such that the following property holds for every solutions $(v,\lambda)\in H^2$ of \eqref{e:lineare} on $[0,2\pi]\times [0,3]$ with $v$ odd:
\begin{itemize}
\item[(a)] If $\mathcal{E} (v, \lambda, 1,2) \geq (1-\eta) \mathcal{E} (v, \lambda,0,1)$ then $\mathcal{E} (v, \lambda,2,3)\geq (1+\eta) \mathcal{E} (v, \lambda,1,2)$.
\end{itemize}
Furthermore there is a positive constant $c_0$ such that the following property holds for every solutions $(v,\lambda)\in H^2$ of \eqref{e:lineare} on $[0,2\pi]\times [0,3]$ with $v$ odd and which satisfies \eqref{e:condizione_extra}:
\begin{itemize}
\item[(b)] If $\mathcal{G} (v, \lambda, 1,2) \geq (1-\eta) \mathcal{G} (v, \lambda,0,1)$ then $\mathcal{G} (v, \lambda,2,3)\geq (1+\eta) \mathcal{G} (v, \lambda,1,2)$.
\end{itemize}
\end{proposition}

\begin{proof} In order to prove claim (a) consider any of the functions $a_k (t)$ and $a_k'' (t)$ and call it $\omega (t)$ and observe we know $k\geq 2$ by assumption. From Proposition \ref{p:odd} and Lemma \ref{l:autovalori} it follows that $\omega$ solves then the ODE
\[
\omega'' (t) - \omega' (t) - c \omega (t) = 0\, ,
\]
where $c$ is a constant which depends on $k$, but it satisfies the bound $c \geq \bar c>0$ for some positive $\bar c$ independent of $k$. The polynomial $x^2 -x - c$ has then a positive and a negative solution $\alpha^+$ and $-\alpha^-$ (also depending on $k$) with $\alpha^\pm \geq \alpha_0 >0$. The function $\omega (t)$ is then given by $D e^{\alpha^+ t} + C e^{-\alpha^- t}$. A simple computations shows that
\[
\frac{d^2}{dt^2} (\omega (t))^2 \geq \hat{c} (\omega (t))^2\, ,
\]
where the positive constant $\hat{c}$ can be chosen to depend on $\alpha_0$ and in particular independent of $k$. Summing the square of all the coefficients involved in the computation of $\mathcal{E}$ we find a non negative function $h(t)$ with the property that $h'' (t) \geq \hat{c} h(t)$ and $\mathcal{E} (v, \lambda, s, \sigma) = \int_s^\sigma h(t)\, dt$. In particualr, $h$ is convex. 
The claim can be thus reduced to, for some $\eta>0$, 
\begin{equation}
\int_1^2 h(t)\, dt \geq (1-\eta) \int_0^1 h(t)\, dt \quad \Longrightarrow \quad \int_2^3 h(t)\, dt \geq (1+\eta) \int_1^2 h(t)\, dt\, .
\end{equation}
Arguing by contradiction, if this were to fail we could find a sequence of convex functions $h_j$ normalized so that $\int_1^2 h_j(t)\, dt =1$ and
\[
\int_1^2 h_j(t)\, dt \geq \max \left\{(1-j^{-1}) \int_0^1 h_j(t)\, dt, (1+j^{-1}) \int_2^3 h_j(t)\, dt\right\} 
\] 
By the convexity of $h_j$ we can extract a subsequence converging locally uniformly to a convex function $h$, which satisfies $h''\geq \hat{c} h$ in the sense of distributions. The latter function $h$ would moreover satisfy 
\[
1 = \int_1^2 h(t)\, dt \geq \max \left\{\int_0^1 {h(t)}\, dt, \int_2^3 {h(t)}\, dt\right\}\, .
\]
Since $h$ is continuous in $(0,3)$ there would then be three points ${0<}s_1<1<s_2{<2}<s_3
{<3}$ such that ${\max \{h(s_1), h(s_3)\}\leq h(s_2)}$. The convexity of $h$ would then imply that $h$ is constant and, since the integral of $h$ over $[1,2]$ is $1$, the constant would have to be $1$. But this would contradict the inequality $h''\geq h$.

Having shown (a) we now turn to (b).  We claim that (b) holds for $c_0$ sufficiently small.
Observe that if $\mathcal{E} (v, \lambda, 1,2)\geq c_0\mathcal{F} (v, \lambda,1,2)$, then (b) is simply implied by (a). 
Thus we may assume $\mathcal{G} (v, \lambda,1,2) = c_0\mathcal{F} (v, \lambda,1,2)$.
We argue by contradiction: for $c_0 = \sfrac{1}{j}$ choose $(v_j, \lambda_j)$ such that 
\[
 \mathcal{G} (v_j, \lambda_j, 1,2) \geq \max \{ (1-\eta) \mathcal{G} (v_j, \lambda_j, 0,1),(1+\eta)^{-1} \mathcal{G} (v_j, \lambda_, 2,3)\}\,.
\]
Using the linearity we can normalize it so that $\mathcal{F} (v_j, \lambda_j,1,2) =1$. Observe that we have the inequalities
\begin{align}
1 = \mathcal{F} (v_j, \lambda_j,1,2) &\geq \max \{ (1-\eta) \mathcal{F} (v_j, \lambda_j, 0,1), (1+\eta)^{-1} \mathcal{F} (v_j, \lambda_j, 2,3)\}\\
1 = \mathcal{F} (v_j, \lambda_j,1,2) &\geq j \max \{\mathcal{E} (v_j, \lambda_j, 1,2),  (1-\eta) \mathcal{E} (v_j, \lambda_j, 0,1), 
(1+\eta)^{-1} \mathcal{E} (v_j, \lambda_j, 2,3)\}\, .\label{e:preponderante}
\end{align}
From Proposition \ref{p:odd} we gain a uniform bound on $\|v_j\|_{H^2 ([0,2\pi]\times [0,3])}
$ and $\|\dot\lambda_j\|_{H^1 ([0,3])}$ and consequently (since $\lambda_j (0) =0$) on $\|\lambda_j\|_{H^2 ([0,3])}$. We then extract a sequence converging weakly to $(v, \lambda)\in H^2$ which satisfies \eqref{e:lineare} and \eqref{e:condizione_extra}. Consider the functions $v$ and $\zeta$, which are the limit of the corresponding maps constructed from $v_j$. From \eqref{e:preponderante} and \eqref{e:def_E} we 
conclude that $\zeta (\phi, t) = a_0 (t) \zeta_0 (\phi) + a_1 (t) \zeta_1 (\phi)$. Unraveling the definition of $\zeta$ we infer 
\[
v (\phi, t) = a_0 (t) (\phi -\pi) \sin \frac{\phi}{2} + \bar{a}_1 (t) \cos \frac{\phi}{2}\, ,
\]
where $\bar{a}_1 (t) = a_1 (t) + \frac{\lambda (t)}{\sqrt{2\pi}}$. 
However the boundary conditions $v (0, t) = v (2\pi, t) =0$ imply $\bar{a}_1 \equiv 0$. We are thus left with the formula $v (\phi, t) = a_0 (t) (\phi -\pi) \sin \frac{\phi}{2}$. Inserting in \eqref{e:lineare} we get:
\begin{equation}
\left\{
\begin{array}{l}
a_0''(t) - a_0' (t) =0\\
\dot\lambda (t) - \ddot\lambda (t) = -\sqrt{2\pi} a_0 (t)\\
\lambda (0) =0\, .
\end{array}\right.
\end{equation}
From the first equation we find $a_0 (t) = c_1 + c_2 e^t$, while from the second we find $\lambda (t) = d_1 - \sqrt{2\pi} c_1 t + d_2 e^t -c_2\sqrt{2\pi} t e^t$, 
i.e. $\lambda (t) = -\sqrt{2\pi} t a_0 (t) + d_1+d_2 e^t$. Using $\lambda (0) =0$ we thus get $\lambda (t) = - \sqrt{2\pi} t a_0 (t) + d (e^t-1)$. We next use \eqref{e:condizione_extra} to derive a further relation between $a_0$ and $\lambda$. The latter reads as
\begin{align}
&({\textstyle{\frac{1}{2}}} a_0 (t) - a'_0 (t)) \int_0^{2\pi}  (\phi-\pi) \sin {\textstyle{\frac{\phi}{2}}} \left(\cos {\textstyle{\frac{3\phi}{2}}} + \cos {\textstyle{\frac{\phi}{2}}}\right)\, d\phi\nonumber\\
+ & a_0 (t) \int_0^{2\pi} \left( \sin {\textstyle{\frac{\phi}{2}}} + {\textstyle{\frac{\phi-\pi}{2}}} \cos {\textstyle{\frac{\phi}{2}}}\right) \left( \sin {\textstyle{\frac{3\phi}{2}}} + \sin {\textstyle{\frac{\phi}{2}}}\right)\, d\phi + \sqrt{{\textstyle{\frac{\pi}{2}}}} \dot\lambda (t) = 0\, . \label{e:extra2}
\end{align}
Observe that, given our formulas for $a_0$ and $\lambda$, the left hand side is linear combination of the functions $1, e^t, te^t$. However, the function $te^t$ appears only in $\dot\lambda (t)$. In particular its coefficient must be $0$. In turn this implies that $a_0$ is constant and $a_0'=0$, which implies that the function $e^t$ would appear only in the $\dot\lambda (t)$ part. We thus conclude that $d=0$ as well. In particular $\lambda (t) = - \sqrt{2\pi} c_1 t$ and $a_0 (t) = c_1$. Finally, since $\int_1^2 (\dot\lambda^2 + a_0^2) =1$, the constant $c_1$ cannot vanish. We thus can, without loss of generality assume $a_0=1$ and $\dot\lambda (t) = -\sqrt{2\pi}$. Taking all this into account, the condition \eqref{e:extra2} can be rewritten as
\[
\frac{1}{2} \int_0^{2\pi} (\phi-\pi) (\sin 2\phi + \sin \phi)\, d\phi + \int_0^{2\pi} \left(\sin^2 {\textstyle{\frac{\phi}{2}}} + \sin {\textstyle{\frac{\phi}{2}}} \sin {\textstyle{\frac{3\phi}{2}}}\right)\, d\phi - \pi = 0\, 
\]
Observe that $\int \sin^2 \frac{\phi}{2} = \pi$, so that the identity can be further simplified into
\[
\underbrace{\frac{1}{2} \int_0^{2\pi} (\phi-\pi) (\sin 2\phi + \sin \phi)\, d\phi}_{=:I} + \underbrace{\int_0^{2\pi} \left(\sin^2 {\textstyle{\frac{\phi}{2}}} \cos \phi\, d\phi + \sin {\textstyle{\frac{\phi}{2}}} \cos {\textstyle{\frac{\phi}{2}}} \sin \phi\right)\, d\phi}_{=:II} = 0\, .
\]
We then compute
\begin{align*}
I &= - \left. (\phi-\pi) \left(\frac{\cos 2\phi}{4} + \frac{\cos \phi}{2}\right)\right|_0^{2\pi} + \int_0^{2\pi} \left(\frac{\cos 2\phi}{4} + \frac{\cos \phi}{2}\right) d\phi
= -\frac{3\pi}{2}\, ,\\
II &= \frac{1}{2} \int_0^{2\pi} ((1-\cos \phi) \cos \phi + \sin^2 \phi)\, d\phi = 0\, ,
\end{align*}
reaching a contradiction.
\end{proof}

\section{Second linearization and proof of Theorem \ref{t:main2}}

The three annuli property of the previous section allows us to improve upon Proposition \ref{p:linearizzazione} and show that the sequence $v_j$ converges indeed on the whole $[0, \infty)$ and that the limit is a decaying solution of the linearized problem.

\begin{proposition}\label{p:linearizzazione-2}
Let $v_j$ and $\lambda_j$ be as in Proposition \ref{p:linearizzazione}, where $T_0$ is fixed to be $1$. Then, there is a pair $(v, \lambda)\in C^2_{loc} ([0,2\pi]\times [0, \infty))$ with $v$ odd and a subsequence, not relabeled, such that $(v_j, \lambda_j)$ converges in $C^2 ([0,2\pi]\times [0, \sigma^{-1}])$ to $(v, \lambda)$ for every $\sigma >0$. Moreover, $(v, \lambda)$ solves \eqref{e:lineare}, satisfies \eqref{e:condizione_extra} and $\|v\|_{C^2 ([k,k+1])} + \|\dot\lambda\|_{C^1 ([k,k+1])} \leq C e^{-\kappa k}$ for some positive universal constants $C$ and $\kappa$ and every $k\in \mathbb N\setminus \{0\}$.
\end{proposition}

Using this second linearization procedure and again the spectral analysis for odd solutions of \eqref{e:lineare} we will then conclude 

\begin{corollary}\label{c:decay_curvature}
There is a constant $\delta_0$ with the following property. Assume $u$ is as in Theorem \ref{t:main2} and $\vartheta$ as in \eqref{e:vartheta}. Then $|\vartheta' (t)| + |\vartheta'' (t)|\leq C e^{-(1+\delta_0) t}$.
\end{corollary}

The latter implies easily Theorem \ref{t:main2}.

\subsection{Proof of Proposition \ref{p:linearizzazione-2}} We start by observing that Proposition \ref{p:linearizzazione} and Proposition \ref{p:tre-anelli} gives easily the following property:
\begin{itemize}
\item If $u$ satisfies the assumptions of Theorem~\ref{t:main2}, $\vartheta, f$ are given by \eqref{e:vartheta} 
 and \eqref{e:f} and $\varepsilon_0$ in Theorem~\ref{t:main2} is sufficiently small, then for every $k\in \mathbb N$ we have
\begin{align}\label{e:Gfo}
&\mbox{if $\mathcal{G} (f^o, \vartheta, k+1,k+2) \geq (1-\frac\eta 2) \mathcal{G} (f^o, \vartheta,k,k+1)$}\notag\\
&\mbox{then $\mathcal{G} (f^o, \vartheta,k+2,k+3)\geq (1+\frac\eta 2) \mathcal{G} (f^o,\vartheta,k+1,k+2)$.}
\end{align}
\end{itemize}
Indeed, assume the claim is false, no matter how small $\varepsilon_0$ in Theorem~\ref{t:main2} is chosen, and 
let thus $f^o_j, \vartheta_j$ be a sequence which violates it when we choose $\varepsilon_0 = \frac{1}{j}$. By translating in the variable $t$ (which just implies a rescaling of 
the variable $r$ in the original problem), we can assume that the claim fails at $k=0$. Additionally, after applying a rotation we can assume 
that $\vartheta_j (0) =0$, so that we can apply Proposition~\ref{p:linearizzazione} with $a=0$ and $T_0=3$. Introduce $\delta_j$, $v_j$ and $\lambda_j$ as in \eqref{e:delta}-\eqref{e:lambda}. Since the functional $\mathcal{G}$ is quadratic, we immediately see that we have 
\begin{equation}\label{e:contra G}
\mathcal{G} (v_j, \lambda_j, 1,2) \geq \max \left\{\left(1-\frac{\eta}{2}\right) \mathcal{G} (v_j, \lambda_j, 0,1), 
\left(1+\frac{\eta}{2}\right)^{-1} \mathcal{G} (v_j, \lambda_j, 2,3)\right\}\, .
\end{equation}
We can further renormalize $\mathcal{G} (v_j, \lambda_j, 1,2)=1$, and since 
\[
C^{-1} (\|v\|_{H^2 ([0,2\pi]\times [\sigma,s])} + \|\dot\lambda\|_{H^1 ([\sigma,s])}) \leq \mathcal{G} (v, \lambda, \sigma,s)
\leq C (\|v\|_{H^2 ([0,2\pi]\times [\sigma,s])} + \|\dot\lambda\|_{H^1 ([\sigma,s])})\,,
\]
we can apply Proposition \ref{p:linearizzazione} to extract a subsequence converging to some $(v, \lambda)$. By the $C^{2,\alpha}$ convergence in $[0,2\pi]\times [1,2]$ for $v_j$ and in $[1,2]$ for $\lambda_j$, we conclude that 
\[
\mathcal{G} (v,\lambda, 1,2) = \lim_j \mathcal{G} (v_j, \lambda_j, 1,2)\, ,
\]
and in particular that the pair $(v, \lambda)$ is nontrivial.

On the other hand the functional $\mathcal{G}$ is lower semicontinuous with respect to the mentioned convergences, 
and we thus infer from \eqref{e:contra G}
\[
\mathcal{G} (v, \lambda, 1,2) \geq \max \left\{\left(1-\frac{\eta}{2}\right) \mathcal{G} (v, \lambda, 0,1), 
\left(1+\frac{\eta}{2}\right)^{-1} \mathcal{G}(v, \lambda, 2,3)\right\}\,,
\]
contradicting Proposition \ref{p:tre-anelli} being $(v,\lambda)$ nontrivial. 

Having completed the proof of \eqref{e:Gfo}, if for some $k_0\in\N$ we were to have
\[
\mathcal{G} (f^o, \vartheta, k_0+1,k_0+2) \geq \left(1-\frac\eta2\right) \mathcal{G} (f^o, \vartheta,k_0,k_0+1)\, ,
\] 
{from \eqref{e:Gfo} we would infer that}
\[
\mathcal{G} (f^o, \vartheta, j,j+1) \geq \left(1+\frac\eta2\right)^{j-(k_0+1)}  \mathcal{G} (f^o, \vartheta,k_0+1,k_0+2) \qquad \forall j \geq k_0+1\, .
\]
However the latter contradicts the fact that $f^o (t, \cdot)$ and $\dot\lambda (t)$ converge smoothly to $0$ for $t\to \infty$.

{We thus conclude that for every $k\in\N$
\[
\mathcal{G} (f^o, \vartheta, k+1,k+2) \leq \left(1-\frac\eta2\right) \mathcal{G} (f^o, \vartheta,k,k+1)\, ,
\]
in turn implying, by iteration,} the existence of positive constants $C$ and $\kappa$ such that
\[
\|f^o\|_{H^2 ([0,2\pi]\times [k,k+1])} + \|\dot\vartheta\|_{H^1 ([k,k+1])} \leq C e^{-\kappa k} \left(\|f^o\|_{H^2 ([0,2\pi]\times [0,1])} + \|\dot\vartheta\|_{H^1 ({[0,1]})}\right)\, .
\]
In turn, if $(v_j, \lambda_j)$ are as in the statement of the proposition, we infer
\[
\|v_j\|_{H^2 ([0,2\pi]\times [k,k+1])} + \|\dot\lambda_j\|_{H^1 ([k,k+1])} \leq C e^{-\kappa k} \left(\|v_j\|_{H^2 ([0,2\pi]\times [0,1])} + \|\dot\lambda_j\|_{H^1 ({[0,1]})}\right) = C e^{-\kappa k}\, .
\]
The conclusion of the proposition is then a simple application of Proposition \ref{p:linearizzazione} with $[0,T_0]$ replaced by $[k,k+1]$, together with an obvious diagonal argument over $k$ and $j$.
 
\subsection{Proof of Corollary \ref{c:decay_curvature}} First of all consider any limit $(v, \lambda)$ as in Proposition \ref{p:linearizzazione-2} and let $\zeta$ be as in \eqref{e:zeta}. By the decay property of $v$ and $\dot \lambda$ we easily infer that
\[
\zeta (\phi, t) = \bar{a}_1 \cos \frac{\phi}{2} + \sum_{k=2}^\infty \bar{a}_k e^{-\mu_k t} \zeta_k (\phi)
\]
where the $\bar{a}_k$ are constants and $-\mu_k$  is the negative solution of the quadratic polynomial $x^2-x-(\nu_k^2-\frac{1}{4})$. Note also that $\bar{a}_1$ is indeed $\lim_{t\to\infty} \lambda (t)$. Recalling Lemma \ref{l:autovalori}, we have $\nu_k \geq \nu_2 > \frac{3}{2}$ when $k\geq 2$ and thus we conclude that $\mu_k \geq \mu_2 > 1$ for all $k\geq 2$. It is then easy to check that we have the estimate
\[
\|v\|_{C^2 ([0,2\pi]\times [T, 2T]} + \|\dot\lambda\|_{C^1 ([T, 2T])} \leq C e^{-\mu_2 T} \left(\|v\|_{H^2 ([0,2\pi]\times[0,1])} + \|\dot\lambda\|_{H^1 ([0,1])}\right)\, ,
\]
where $C$ is a constant independent of $T$. Fix now $T$. Using Proposition \ref{p:linearizzazione-2} we then conclude that, if $u$ is as in Theorem \ref{t:main2} and $\vartheta$ and $f$ as in Lemma \ref{l:nonlinear} and $\varepsilon_0$ sufficiently small (depending on $T$), then
\begin{align*}
\|f^o\|_{C^2 ([0,2\pi]\times [T, 2T]} + \|\dot\vartheta\|_{C^1 ([T, 2T])} &\leq 2 C e^{-\mu_2 T}\left(\|f^o\|_{H^2 ([0,2\pi]\times[0,1])} + \|\dot\vartheta\|_{H^1 ([0,1])}\right)\\
&\leq \bar{C} e^{-\mu_2 T}\left(\|f^o\|_{C^2 ([0,2\pi]\times[0,T])} + \|\dot\vartheta\|_{C^1 ([0,T])}\right)\, ,
\end{align*}
where the constant $\bar C$ is independent of $T$.
By a simple rescaling argument, this actually implies that
\begin{align*}
&\|f^o\|_{C^2 ([0,2\pi]\times [(k+1) T, (k+2)T]} + \|\dot\vartheta\|_{C^1 ([(k+1)T, (k+2) T])} \\
\leq & \bar{C} e^{-\mu_2 T}\left(\|f^o\|_{C^2 ([0,2\pi]\times[kT,(k+1)T])} + \|\dot\vartheta\|_{C^1 ([kT,(k+1)T])}\right)\, \qquad \forall k\in \mathbb N\, .
\end{align*}
Considering now that $\mu_2>1$, while the constant $\bar C$ is independent of $T$, we can choose the latter large enough so that
${\bar{C}} e^{-\mu_2 T} = e^{-(1+\delta_0) T}$ for some positive $\delta_0$. We then can iterate the latter inequality to infer
\begin{align*}
& \|f^o\|_{C^2 ([0,2\pi]\times [(k+1) T, (k+2)T]} + \|\dot\vartheta\|_{C^1 ([(k+1)T, (k+2) T])}\\  
\leq &e^{-(1+\delta_0) kT}\left(\|f^o\|_{C^2 ([0,2\pi]\times[0,T])} + \|\dot\vartheta\|_{C^1 ([0,T])}\right)\, .
\end{align*}
This easily gives the conclusion of the corollary.

\subsection{Proof of Theorem \ref{t:main2}} Using the relation $r= e^{-t}$ and \eqref{e:formula_curvatura}, \eqref{e:main_estimate} is an obvious consequence of Corollary \ref{c:decay_curvature}.



\end{document}